\documentclass[12pt]{article}
\usepackage[T1]{fontenc}
\usepackage{lmodern,amsmath,amsthm,amsfonts,amssymb,graphicx,float,microtype,underscore,mathtools,xurl,thmtools,thm-restate}
\usepackage{stmaryrd}
\usepackage[dvipsnames,svgnames,table]{xcolor}
\usepackage[shortlabels]{enumitem}
\setlist[enumerate]{topsep=0ex,itemsep=0ex,parsep=1ex}
\setlist[itemize]{topsep=0ex,itemsep=0ex,parsep=1ex}
\usepackage[unicode=true]{hyperref}
\hypersetup{ 
colorlinks,
breaklinks=true,
linkcolor={blue!60!black},
citecolor={black},
urlcolor={blue!60!black},
pdftitle={Optimal Tree-Decompositions with Bags of Bounded Pathwidth}}
\usepackage[capitalise, compress, nameinlink, noabbrev]{cleveref}
\crefname{lem}{Lemma}{Lemmas}
\crefname{thm}{Theorem}{Theorems}
\crefname{cor}{Corollary}{Corollaries}
\crefname{open}{Open Problem}{Open Problems}

\newcommand{\defn}[1]{\textcolor{Maroon}{\emph{#1}}}
\newcommand{\mathdefn}[1]{\textcolor{Maroon}{#1}}
\usepackage[longnamesfirst,numbers,sort&compress]{natbib}
\makeatletter
\def\NAT@spacechar{~}
\makeatother
\usepackage[tmargin=30mm,bmargin=30mm,lmargin=30mm,rmargin=30mm]{geometry}
\renewcommand{\baselinestretch}{1.09}
\allowdisplaybreaks
\renewcommand{\epsilon}{\varepsilon}
\renewcommand{\emptyset}{\varnothing}

\renewcommand{\geq}{\geqslant}
\renewcommand{\leq}{\leqslant}

\DeclareMathOperator{\dist}{dist}

\DeclareMathOperator{\tw}{tw}

\DeclareMathOperator{\pw}{pw}

\DeclareMathOperator{\td}{td}

\newcommand{\GG}{\mathcal{G}}

\newcommand{\DD}{\mathcal{D}}
\renewcommand{\thefootnote}{\fnsymbol{footnote}}
\theoremstyle{plain}
\newtheorem{thm}{Theorem}
\newtheorem{lem}[thm]{Lemma}
\newtheorem{cor}[thm]{Corollary}
\newtheorem{obs}[thm]{Observation}

\crefname{obs}{Observation}{Observations}
\newtheorem*{lem*}{Lemma}
\theoremstyle{definition}

\newtheorem*{conj*}{Conjecture}
\theoremstyle{remark}
\newtheorem{clms}{Claim}
\newtheorem*{claim}{Claim}
\usepackage{etoolbox}
\newenvironment{proofclaim}[1][]
	{\vspace{-\baselineskip}\begin{proof}[Proof] }{\end{proof}}

\begin{document}
\title{\bf\boldmath\fontsize{18pt}{18pt}\selectfont 
Optimal Tree-Decompositions with\\
Bags of Bounded Pathwidth}

\author{%
Kevin Hendrey\,\footnotemark[2] \quad  
Robert Hickingbotham\,\footnotemark[3] \\ 
J\k{e}drzej Hodor\,\footnotemark[4] \quad  
David~R.~Wood\,\footnotemark[2]
}

\maketitle

\begin{abstract}
We show that every planar graph has a tree-decomposition with optimal width such that the subgraph induced by each bag has pathwidth at most 3. This bound is best possible, and for tree-decompositions that satisfy a certain minimality condition, we in fact give a precise description of the possible structures in each bag. Moreover, we show that the union of any $k$ bags has pathwidth $O(k)$. We also show that graphs excluding a fixed double-apex-forest minor have a tree-decomposition with optimal width such that the subgraph induced by each bag has bounded pathwidth. This includes graphs embeddable on any fixed surface. As a byproduct of our machinery, we give a new proof of the linear grid minor theorem for planar graphs. 
\end{abstract}

\footnotetext[2]{School of Mathematics, Monash University, Melbourne, Australia (\texttt{\{Kevin.Hendrey1,david.wood\}@monash.edu}). Research supported by the Australian Research Council. Research of Wood also supported by NSERC. }

\footnotetext[3]{D\'epartement d'Informatique, Universit\'e libre de Bruxelles, Belgium ({\tt robert.hickingbotham@ulb.be}). Research supported by the Belgian National Fund for Scientific Research (FNRS).}

\footnotetext[4]{Theoretical Computer Science Department, 
Faculty of Mathematics and Computer Science and  Doctoral School of Exact and Natural Sciences, Jagiellonian University, Krak\'ow, Poland (\texttt{jedrzej.hodor@gmail.com}). Research supported by a Polish Ministry of Education and Science grant (Per\l{}y Nauki; PN/01/0265/2022). }

\renewcommand{\thefootnote}{\arabic{footnote}}


\section{Introduction}  
\label{Intro}

Treewidth and tree-decompositions are of central importance to both structural and algorithmic graph theory.
While treewidth is defined in terms of optimal tree-decompositions, the study of tree-decompositions has diverged from the study of treewidth in recent years, with many parameters arising that can be defined in terms of the structure of the bags of tree-decompositions that are not necessarily optimal. This paper builds upon the recent work of Hendrey and Wood \cite{HW26} to further our understanding of the structure of the bags of optimal tree-decompositions.

For a non-null tree $T$, a \defn{$T$-decomposition} of a graph\footnote{We consider simple undirected graphs $G$ with vertex set $V(G)$ and edge set $E(G)$. A graph $H$ is a \defn{minor} of a graph $G$ if a graph isomorphic to $H$ can be obtained from $G$ by a sequence of edge deletions, vertex deletions, and edge contractions. 
A \defn{graph class} is a collection of graphs closed under isomorphism. A graph class $\GG$ is \defn{minor-closed} if for every graph $G\in\GG$ every minor of $G$ is in $\GG$. A graph class $\GG$ is \defn{monotone} if for every graph $G\in\GG$ every subgraph of $G$ is in $\GG$. A graph class $\GG$ is \defn{proper} if some graph is not in $\GG$.} $G$ is a collection $(B_x:x \in V(T))$ such that:
\begin{itemize}
    \item $B_x\subseteq V(G)$ for each $x\in V(T)$,
    \item for each edge ${vw \in E(G)}$, there exists a node ${x \in V(T)}$ with ${v,w \in B_x}$, and
    \item for each vertex ${v \in V(G)}$, the set $\{ x \in V(T) : v \in B_x \}$ induces a non-empty (connected) subtree of $T$.
\end{itemize}
The \defn{width} of such a $T$-decomposition is ${\max\{ |B_x| : x \in V(T) \}-1}$. A \defn{tree-decomposition} is a $T$-decomposition for any tree $T$. The \defn{treewidth} of a graph $G$, denoted \defn{$\tw(G)$}, is the minimum width of a tree-decomposition of $G$. A tree-decomposition of a graph $G$ with width $\tw(G)$ is said to be \defn{optimal}. For the sake of brevity, if $S$ is a set of vertices in a graph $G$, define the \defn{treewidth} of $S$ to be $\tw(G[S])$. Treewidth is the standard measure of how similar a graph is to a tree. Indeed, a connected graph has treewidth at most 1 if and only if it is a tree. See  \citep{HW17,Bodlaender98,Reed97} for surveys on treewidth. 

A \defn{path-decomposition} is a $P$-decomposition for any path $P$. The \defn{pathwidth} of a graph $G$, denoted \defn{$\pw(G)$}, is the minimum width of a path-decomposition of $G$. By definition, $\tw(G)\leq\pw(G)$. Note that pathwidth can be much larger than treewidth. For example, the complete binary tree with $n$ vertices has treewidth 1 and pathwidth $\Theta(\log n)$. 

In addition to studying the width of tree-decompositions, much recent work has studied tree-decompositions where a given graph parameter\footnote{A graph \defn{parameter} $\beta$ takes as input a graph $G$ and outputs a real number $\beta(G)$, such that $\beta(G_1)=\beta(G_2)$ for all isomorphic graphs $G_1$ and $G_2$. 
A graph parameter $\beta$ is \defn{minor-monotone} if $\beta(H)\leq\beta(G)$ for every graph $G$ and every minor $H$ of $G$. Two graph parameters $\beta_1$ and $\beta_2$ are \defn{tied} if there is a function $f$ such that $\beta_1(G)\leq f(\beta_2(G))$ and $\beta_2(G)\leq f(\beta_1(G))$ for every graph $G$.} is bounded on the subgraph induced by each bag. Such parameters include chromatic number~\citep{Seymour16,HK17,BFMMSTT19,HRWY21,KKLRU},
diameter~\citep{DG07,BS24,NSS-AS-I,Hickingbotham25,DK25}, 
independence number~\citep{Yolov18,DMS24a}, and of most relevance to this paper, treewidth~\citep{LNW,HW26}, pathwidth~\citep{LNW}, and treedepth~\citep{LNW}\footnote{The \defn{treedepth} of a graph $G$, denoted $\mathdefn{\td(G)}$, is defined recursively as follows: if $|V(G)|=1$ then $\td(G):=1$; if $G$ is disconnected, then $\td(G)$ is the maximum treedepth of a component of $G$; and if $G$ is connected and $|V(G)|\geq 2$, then $\td(G):=\min\{1+\td(G-v):v\in V(G)\}$. It is well known that $\tw(G)\leq\pw(G)\leq\td(G)-1$ for every graph $G$.}.

For example, \citet{LNW} showed that graphs in any proper minor-closed class have tree-decompositions with bags of bounded pathwidth. In particular, every planar graph has a tree-decomposition in which every bag has pathwidth at most 3. This bound is best possible even for bags of given treewidth whenever the graph $G$ contains $K_4$, since in any tree-decomposition of $G$ each clique of $G$ is contained in a single bag, and $\tw(K_4)=\pw(K_4)=3$. We emphasise that in these results, there is no bound on the width of the tree-decomposition. Indeed the width can be $\Omega(|V(G)|)$. 

\citet{HW26} pushed this direction further by considering tree-decompositions  with optimal width and simultaneously with bags of small treewidth. For example, they showed that every planar graph has an optimal tree-decomposition in which every bag has treewidth at most 3. More generally, \citet{HW26} showed that graphs in any proper minor-closed class admit tree-decompositions that simultaneously have optimal width and every bag has bounded treewidth. 

This paper pushes this direction further by studying optimal tree-decompositions with bags of bounded pathwidth. 
Our first result strengthens all of the above results for planar graphs.

\begin{restatable}{thm}{PlanarOptimalPW}
\label{PlanarOptimalPW3bags}
    Every planar graph has an optimal tree-decomposition in which every bag has pathwidth at most 3. 
\end{restatable}

\cref{PlanarOptimalPW3bags} is best possible in several ways. Firstly, the bound of 3 in \cref{PlanarOptimalPW3bags} cannot be improved for any planar graph $G$ that contains $K_4$, since in any tree-decomposition of $G$, each clique of $G$ appears in some bag, and $K_4$ has pathwidth 3. Secondly, `bounded pathwidth' in \cref{PlanarOptimalPW3bags} cannot be replaced by `bounded treedepth'. This is because \citet[Lemma~3.38]{LNW} showed that every tree-decomposition of a triangulated $n\times n$ grid has a bag that contains an $n$-vertex path, and thus has treedepth $\Omega(\log n)$. In fact, \citet{LNW} proved that for a minor-closed class $\GG$, every graph in $\GG$ has a tree-decomposition in which each bag has bounded treedepth if and only if $\GG$ has bounded treewidth.

A key ingredient in the proof of \cref{PlanarOptimalPW3bags} is the notion of a `refined tree-decomposition'~\citep{HW26}, which may be viewed as a type of minimal tree-decomposition; see \cref{TreeDecompositions} for the definition. We characterise the structure of the bags in refined tree-decompositions of an arbitrary graph (\cref{thm:refinedcharacterisation}). This leads to a precise structural characterisation of the bags in a refined tree-decomposition of a planar triangulation (\cref{thm:TriangDetailed}), from which \cref{PlanarOptimalPW3bags} follows.

We extend \cref{PlanarOptimalPW3bags} in multiple ways, first considering the pathwidth of a union of bags (see \cref{ManyBags}). 

\begin{restatable}{thm}{PlanarOptimalManyBagsStatement}
\label{PlanarOptimalManyBags}
    Every planar graph $G$ has an optimal tree-decomposition $(B_x:x\in V(T))$ such that $\pw(G[B_x])\leq3$ for each $x\in V(T)$, and for any $X\subseteq V(T)$ with $|X|\geq 2$ we have $\pw(G[\bigcup\{B_x:x\in X\}])\leq 28|X|-24$.
\end{restatable}

Next consider graphs embeddable in a fixed surface. \citet{LNW} proved that every graph with Euler genus\footnote{The \defn{Euler genus} of an orientable surface with $h$ handles is $2h$. The \defn{Euler genus} of a non-orientable surface with $c$ cross-caps is $c$. The \defn{Euler genus} of a graph $G$ is the minimum Euler genus of a surface in which $G$ embeds (with no crossings).} $g$ has a tree-decomposition in which every bag has treewidth $O(g)$, and this bound is best possible. \citet{HW26} strengthened this result, by showing that every graph with Euler genus $g$ has an optimal tree-decomposition in which every bag has treewidth $O(g)$. We improve this result to bounded pathwidth, albeit with a worse bound. 

\begin{thm}
\label{EulerGenusOptimalTWgbags}
    For every integer $g\geq 0$ there exists $c$ such that every graph with Euler genus $g$ has an optimal tree-decomposition in which every bag has pathwidth at most $c$.
\end{thm}

\cref{EulerGenusOptimalTWgbags} is proved in \cref{Extensions}. In fact, this result holds for graphs excluding any fixed double-apex-forest minor. Here a graph $H$ is \defn{double-apex-forest} if $H-v-w$ is a forest for some non-adjacent vertices $v,w\in V(H)$. The method introduced in this proof is very general, and applies in other settings of interest. In particular, for a broad collection  of minor-monotone graph parameters $\beta$ (which includes treedepth, pathwidth and treewidth), we determine which minor-closed graph classes $\mathcal{C}$ have the property that every bag of every refined tree-decomposition of a graph from $\mathcal{C}$ has bounded $\beta$. 

We finish this introduction with the natural open problem arising from this work: Do graphs in any proper minor-closed class have an optimal tree-decomposition with bags of bounded pathwidth? 

\section{Preliminaries}

We consider simple finite undirected graphs $G$ with vertex-set $V(G)$ and edge-set $E(G)$.
For each vertex $v\in V(G)$, let $\mathdefn{N_G(v)}:=\{w\in V(G):vw\in E(G)\}$ and $\mathdefn{N_G[v]} := N_G(v)\cup\{v\}$. 
Let $\mathdefn{\overline{G}}$ denote the complement of $G$, and let $\mathdefn{G+H}$ denote the graph obtained from the disjoint union of graphs $G$ and $H$ by adding an edge between each vertex of $G$ and each vertex of $H$. 

The operation of \defn{contracting} an edge $vw$ of $G$ consists of deleting $v$ and $w$ from $G$ and adding a new vertex $z$ adjacent to $N_G(v)\cup N_G(w) \setminus \{v,w\}$.
A \defn{minor} of $G$ is any graph that is isomorphic to a graph that can be obtained from a subgraph of $G$ by contracting edges.
A \defn{model} of a graph $H$ in $G$ is a collection $(M_v:v\in V(H))$ of pairwise disjoint subsets of $V(G)$, called \defn{branch-sets}, such that for each $v\in V(H)$ the graph $G[M_v]$ is connected, and for each $vw\in E(H)$ there is an edge from $M_v$ to $M_w$.
It is easy to see that a graph isomorphic to $H$ can be obtained from $G[\bigcup\{M_v:v\in V(H)\}]$ by contracting the edges of a spanning subtree of each graph in $\{G[M_v]:v\in V(H)\}$, and deleting edges $vw\in E(G)$ with $v\in M_a$ and $w\in M_b$ where $ab\in E(\overline{H})$. Likewise, if $H$ is obtained from a subgraph $H'$ of $G$ by contracting all edges in $E'\subseteq E(H')$, then the components of the spanning subgraph of $H'$ with edge set $E'$ form a minor model of $H$ in $G$.
Thus graph minors can be equivalently defined in terms of models.
We write $\mathdefn{H\preccurlyeq G}$ to mean $H$ is a minor of $G$.

For graphs $A$ and $B$, the \defn{Cartesian product $A\square B$} is the graph with vertex-set $V(A)\times V(B)$ where $(x,y_1)(x,y_2)$ is an edge of $A\square B$ for each $x\in V(A)$ and each $y_1y_2\in E(B)$, and
$(x_1,y)(x_2,y)$ is an edge of $A\square B$ for each $x_1x_2\in E(A)$ and each $y\in V(B)$. 

See \citep{Diestel5} for any undefined graph-theoretic terminology. 

\section{Tree-Decompositions}
\label{TreeDecompositions}

This section presents results about normal, atomic and refined tree-decompositions, starting with results from the literature, moving on to some new results that will be used in later sections.

Let $(B_x:x\in V(T))$ be a tree-decomposition of a graph $G$. If $B_x\subseteq B_y$ for some edge $xy\in E(T)$, then let $T'$ be the tree obtained from $T$ by contracting $xy$ into a new vertex $z$, and let $B_z:=B_y$. Then $(B_x:x\in V(T'))$ is a tree-decomposition of $G$ with width equal to the width of $(B_x:x\in V(T))$, and $|V(T')|<|V(T)|$. To \defn{normalise} a tree-decomposition means to apply this operation until $B_x\not\subseteq B_y$ for each edge $xy\in E(T)$. This process terminates, since each contraction decreases $|V(T)|$ by exactly $1$. A tree-decomposition $(B_x:x\in V(T))$ with $B_x\not\subseteq B_y$ for each $xy\in E(T)$ is said to be \defn{normal}. Applying normalisation to an optimal tree-decomposition of a graph $G$ gives a tree-decomposition of $G$ that is optimal and normal.

Say $(B_x:x\in V(T))$ is a normal tree-decomposition of a graph $G$ with $V(G)\neq\emptyset$. Root $T$ at an arbitrary node $r\in V(T)$. For each edge $xy\in E(T)$ where $y$ is the parent of $x$, let $f(xy)$ be any vertex in $B_x\setminus B_y$. Thus $f$ is an injection from $E(T)$ to $V(G-B_r)$. So $|V(G)|-1\geq |V(G)|-|B_r|\geq|E(T)|=|V(T)|-1$, and $|V(T)|\leq|V(G)|$. That is, every normal tree-decomposition of a graph $G$ has at most $|V(G)|$ bags, which is a well-known fact. 

For a tree-decomposition $\DD=(B_x:x\in V(T))$ of a graph $G$ and a positive integer $i$, let $\mathdefn{n_i(\DD)}$ be the number of nodes $x\in V(T)$ such that $|B_x|=i$, and let 
\[\mathdefn{N(\DD)} :=(n_{|V(G)|}(\DD),n_{|V(G)|-1}(\DD),\dots,n_1(\DD)).\]
Sometimes $N(\DD)$ is called the \defn{fatness} of $\DD$. A tree-decomposition $\DD$ is \defn{atomic} if there is no tree-decomposition $\DD'$ of $G$ such that $N(\DD')$ is lexicographically less than $N(\DD)$. This concept was essentially introduced by \citet{Thomas90} and has since been studied by various authors~\citep{Muller12,DM16,DM18,EW19,AACHS25,CDN16,CGHLS24,Weissauer19}. The following lemma is implicit in the work of \citet{Thomas90}, and explicitly proven by \citet{HW26}.

\begin{lem}[{\protect\citep[Lemma 5]{HW26}}] 
\label{lem:atomic}
    Every graph has an atomic tree-decomposition, and every atomic tree-decomposition is normal.
\end{lem}

\citet{HW26} defined a tree-decomposition $\DD'$ of a graph $G$ to be a \defn{refinement} of a tree-decomposition $\DD$ of $G$ if each bag of $\DD'$ is a subset of a bag of $\DD$. If $\DD'$ is a refinement of $\DD$ and not vice versa, then $\DD'$ is a \defn{proper refinement} of $\DD$. A tree-decomposition is \defn{refined} if it is normal and has no proper refinement.

\begin{lem}[{\protect\citep[Lemma 6]{HW26}}] 
\label{lem:refinement}
    Every atomic tree-decomposition is refined.
\end{lem}

\cref{lem:atomic,lem:refinement} imply that every graph has a refined tree-decomposition, which we henceforth use implicitly.

\begin{lem}[{\protect\citep[Lemma 7]{HW26}}] 
\label{ComponentNormalisation}
    For every refined tree-decomposition $\DD=(B_x:x\in V(T))$ of a graph $G$, for all nodes $x,y\in V(T)$ (not necessarily distinct), for any subset $S\subseteq B_y$, the bag $B_x$ intersects exactly one component of $G-S$, unless $x=y$ and $S=B_y$.
\end{lem}

A \defn{separation} of a graph $G$ is an ordered pair $(A,B)$ such that $A,B\subseteq V(G)$ and $G=G[A]\cup G[B]$. The \defn{order} of $(A,B)$ is $|A\cap B|$. A separation $(A,B)$ of a graph $G$ \defn{breaks} a set $S\subseteq V(G)$ if $S\setminus A$ and $S\setminus B$ are both nonempty, and $A\cap B\subseteq S$. 
A set $S\subseteq V(G)$ is \defn{breakable} if there exists a separation $(A,B)$ that breaks $S$, otherwise $S$ is \defn{unbreakable}.

\begin{lem}[{\protect\citep[Lemma 8]{HW26}}] 
\label{RefinedUnbreakable}
Every bag of a refined tree-decomposition is unbreakable. 
\end{lem}

The remainder of this section presents new results about refined tree-decompositions.

\begin{lem}\label{cor:adhesionsets}
    Let $\mathcal{D}:=(B_x:x\in V(T))$ be a refined tree-decomposition of a graph $G$. For any $x\in V(T)$ and for any pair $v,w$ of non-adjacent vertices in $B_x$, there exists $y\in N_T(x)$ such that $\{v,w\}\subseteq B_x\cap B_y$.
\end{lem}
\begin{proof}
    Let $S:=B_x\setminus \{v,w\}$.
    By \cref{ComponentNormalisation} applied with $x = y$, there is a component $C$ of $G-S$ containing both $v$ and $w$.
    Let $P$ be a path from $v$ to $w$ in $C$, and let $P'=P-\{v,w\}$. Since $v$ and $w$ are non-adjacent, $V(P')$ is non-empty.
    Furthermore, $P'$ is connected and is vertex disjoint from $B_x$.
    It follows that there is a component $T'$ of $T-x$ that contains all nodes $y\in V(T)$ such that $B_y$ intersects $V(P')$.
    Let $y$ be the neighbour of $x$ in $V(T')$.
    Since $v$ and $w$ both have a neighbour in $V(P')$, we must have $\{v,w\}\subseteq B_y$, as required. 
\end{proof}

Note that a very similar lemma to \cref{cor:adhesionsets} in the more restrictive setting of atomic tree-decompositions was proved by \citet[Lemma~3.5]{Muller12} (see \citep{DM18} for the final published version).

\begin{thm}\label{thm:refinedcharacterisation}
A normal tree-decomposition $\mathcal{D}=(B_x:x\in V(T))$ of a graph $G$ is refined if and only if for every $x\in V(T)$ and every $v,w\in B_x$ there is a path from $v$ to $w$ with no internal vertex in $B_x$.
\end{thm}

\begin{proof}
    First, suppose that $\mathcal{D}$ is refined, and consider an arbitrary $x\in V(T)$ and $v,w\in B_x$.
    By \cref{ComponentNormalisation} with $y=x$ and $S=B_x\setminus \{v,w\}$, $v$ and $w$ are in the same component of $G-S$, meaning there is a path from $v$ to $w$ with no internal vertex in $B_x$, as required.

    Now instead suppose that for every $x\in V(T)$ and every $v,w\in B_x$ there is a path from $v$ to $w$ with no internal vertex in $B_x$.
    Suppose for contradiction that $\mathcal{D}$ is not refined. 
    Let $\mathcal{D}'=(B'_x:x\in V(T'))$ be a proper refinement of $\mathcal{D}$, and let $x^*\in V(T)$ be such that $B_{x^*}$ is not a subset of $B'_y$ for any $y\in V(T')$.
    For each $v\in V(G)$, let $T'_v$ be the subtree of $T'$ induced by the set of nodes $x$ such that $v\in B'_x$.
    By Helly's Theorem, if $T'_v$ and $T'_w$ intersect for every $v$ and $w$ in $B_{x^*}$, then there is some $y^*$ in $\bigcap_{v\in B_{x^*}}V(T'_v)$. But this means $B_{x^*}\subseteq B'_{y^*}$, contradicting our assumption.
    Hence, there exist $v,w\in B_{x^*}$ such that there is no bag in $\mathcal{D}'$ containing both of them.
    In particular, $v$ and $w$ are non-adjacent.
    Let $P'$ be a minimal path from $V(T'_v)$ to $V(T'_w)$ in $T'$, and let $x'$ and $y'$ be a pair of adjacent vertices on this path, and let $S:=B'_{x'}\cap B'_{y'}$.
    Since $T'_v$ and $T'_w$ are disjoint, $S\cap \{v,w\}=\emptyset$, and by the properties of tree-decompositions $v$ and $w$ are in distinct components of $G-S$.
    Since $\mathcal{D}'$ is a refinement of $\mathcal{D}$, there is some $z\in V(T)$ such that $B_z\supseteq B'_{x'}\supseteq S$.
    Thus, there are vertices $v,w\in V(G)$ and a set $S\subseteq V(G)\setminus\{v,w\}$ such that $S$ separates $v$ from $w$ in $G$ and there are bags $B_{x^*}$ and $B_z$ of $\cal D$ containing $\{v,w\}$ and $S$ respectively.
    Subject to these conditions, let $v$, $w$, $S$, $x^*$ and $z$ be chosen to minimise $\dist_{T}(x^*,z)$.
    Note that every path from $v$ to $w$ has an internal vertex in $S$, so our assumption on $\DD$ implies that $x^*\neq z$.
    
    Let $Q$ be a path from $v$ to $w$ in $G$ that is internally disjoint from $B_{x^*}$, and let $Q'=Q-\{v,w\}$. Note that $V(Q')$ contains a vertex in $S$, and so is non-empty.
    Furthermore, $Q'$ is connected and is vertex disjoint from $B_{x^*}$.
    It follows that there is a component $T^*$ of $T-x^*$ that contains all nodes $y\in V(T)$ such that $B_y$ intersects $V(Q')$.
    In particular, this component contains $z$, since $V(Q')$ intersects $S$.
    Let $y$ be the neighbour of $x^*$ in $V(T^*)$.
    Since $v$ and $w$ both have a neighbour in $V(Q')$, we must have $\{v,w\}\subseteq B_y$. But $\dist_T(y,z)\leq \dist_T(x^*,z)-1$, contradicting the choice of $x^*$ and $z$.
    This contradiction completes the proof.   
\end{proof}


    


    

For a tree $T$ and edge $xy\in E(T)$, let \defn{$T_{x:y}$} and \defn{$T_{y:x}$} be the subtrees of $T-xy$ respectively containing $x$ and $y$. For a tree-decomposition $(B_x:x\in V(T))$ of a graph $G$ and for each edge $xy\in E(T)$, let $\mathdefn{G_{x:y}}:= G[ \bigcup\{ B_z\setminus B_y :z\in V(T_{x:y}) \}]$. 

For a graph $G$ and a set $S\subseteq V(G)$, let \defn{$G\llbracket S\rrbracket$} be the graph obtained from $G$ by contracting each edge of $G-S$.
The vertices in $V(G\llbracket S\rrbracket)\setminus S$ are called the \defn{external} vertices of $G\llbracket S\rrbracket$. There is one external vertex for each component of $G-S$. 

\begin{lem}\label{lem:torso}
    For any refined tree-decomposition $\mathcal{D}$ of a graph $G$, for each bag $B$ of $\mathcal{D}$, the graph $G\llbracket B\rrbracket$ has the following properties:
        \begin{enumerate}[(1)]
            \item\label{torso1} for every pair of non-adjacent vertices $v$ and $w$ in $B$, there is an external vertex of $G\llbracket B\rrbracket$ adjacent to $v$ and $w$, 
            \item\label{torso2} no external vertex of $G\llbracket B\rrbracket$ is adjacent to every vertex in $B$, and
            \item\label{torso3} for every external vertex $x$ of $G\llbracket B\rrbracket$, the graph $G[N_{G\llbracket B\rrbracket}(x)]+\overline{K_2}$ is a minor of $G\llbracket B\rrbracket$.
        \end{enumerate}
\end{lem}
\begin{proof}
Say $\mathcal{D}=(B_t:t\in V(T))$. Consider a pair of non-adjacent vertices $v$ and $w$ in $B$. By \cref{thm:refinedcharacterisation}, there is a path from $v$ to $w$ in $G$ with no internal vertex in $B$. The internal vertices of this path are contracted into a single vertex of $G\llbracket B\rrbracket$, which proves~\ref{torso1}.

    Consider an arbitrary external vertex $x$. Let $C$ be the component of $G-B$ corresponding to $x$, and let $t\in V(T)$ satisfy $B=B_t$. Since $C$ is connected and disjoint from $B_t$, there is a component $T'$ of $T-t$ such that $\{y\in V(T)\colon B_y\cap V(C)\neq \emptyset\}\subseteq V(T')$. Let $t'$ be the unique neighbour of $t$ in $T'$. By the properties of a tree-decomposition, every vertex in $B_t$ having a neighbour in $C$ belongs to $B_{t'}$. Since $\mathcal{D}$ is normal, $B_t \not \subseteq B_{t'}$. Hence some vertex of $B$ has no neighbour in $C$, and therefore $x$ is not adjacent to every vertex in $B$.

    To prove \ref{torso3}, consider an arbitrary external vertex $x$ and note that by \ref{torso2} there is a vertex $v$ in $B\setminus N_{G\llbracket B\rrbracket}(x)$.
    For every non-neighbour $w$ of $v$ in $N_{G\llbracket B\rrbracket}(x)$, there is an external vertex $x_w$ adjacent to both $v$ and $w$ by \ref{torso1}.
    Thus we find $G[N_{G\llbracket B\rrbracket}(x)]+\overline{K_2}$ as a minor of $G\llbracket B\rrbracket$ by contracting every external vertex adjacent to $v$ into $v$ and then deleting all vertices not in $\{x,v\}\cup N_{G\llbracket B\rrbracket}(x)$. Since $G\llbracket B\rrbracket$ is a minor of $G$, the result follows.
\end{proof}

\section{Planar Graphs}
\label{Planar}

This section proves \cref{PlanarOptimalPW3bags} showing that every planar graph has an optimal tree-decomposition with bags of pathwidth 3. We start by recounting the proof of \citet{HW26} showing that every planar graph has an optimal tree-decomposition with bags of treewidth 3, which employs the following definitions introduced by \citet{DF21}. A planar graph $G$ is \defn{separable} with respect to a plane embedding $\Pi$ of $G$ if there is a cycle $C$ in $G$ such that there is a vertex of $G-V(C)$ in the interior of $C$ with respect to $\Pi$, and there is a vertex of $G-V(C)$ in the exterior of $C$ with respect to $\Pi$. Otherwise, $G$ is \defn{non-separable} with respect to $\Pi$. That is, for any cycle $C$ in $G$, all the vertices of $G-V(C)$ are in the interior of $C$, or all the vertices of $G-V(C)$ are in the exterior of $C$. 

Let $\Pi$ be a plane embedding of a planar graph $G$. Suppose that $S\subseteq V(G)$ and $G[S]$ is separable with respect to the plane embedding of $G[S]$ induced by $\Pi$. So there is a cycle $C$ in $G[S]$ with at least one vertex of $S$ in the interior of $C$, and at least one vertex of $S$ in the exterior of $C$. Let $A$ be the set of all vertices of $G$ in $C$ or in the interior of $C$. Let $B$ be the set of all vertices of $G$ in $C$ or in the exterior of $C$. By the Jordan Curve Theorem, $(A,B)$ is a separation of $G$ that breaks $S$. So if a set $S\subseteq V(G)$ is unbreakable, then $S$ is non-separable with respect to $\Pi$. \cref{RefinedUnbreakable} thus implies:

\begin{lem}[{\protect\citep[Lemma 13]{HW26}}]
\label{PlanarNonSep}
Let $\Pi$ be a plane embedding of a planar graph $G$. For every bag $B$ of every refined tree-decomposition of $G$, $G[B]$ is non-separable with respect to the embedding of $G[B]$ induced by $\Pi$.
\end{lem}

A planar graph $G$ is \defn{non-separable} if $G$ is non-separable with respect to some plane embedding of $G$. \citet{DF21} showed that the class of non-separable planar graphs is minor-closed, and that a graph $G$ is non-separable planar if and only if $G$ does not contain $K_1 \cup K_4$ or $K_1 \cup K_{2,3}$ or $K_{1,1,3}$ as a minor. In fact, \citet{DF21} provided the following precise structural characterisation: any non-separable planar graph is either outerplanar, or a subgraph of a \defn{wheel} (a graph obtained from a cycle by adding a universal vertex), or a subgraph of an \defn{elongated triangular prism} (a graph obtained from the triangular prism $K_3\square K_2$ by subdividing edges that are not in triangles any number of times). 

\begin{figure}[!ht]
(a) \hspace*{-4mm} \includegraphics{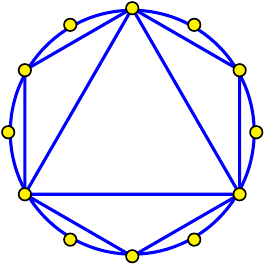}
\quad 
(b) \hspace*{-4mm}\includegraphics{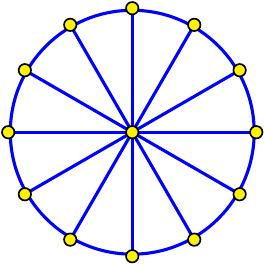}
\quad 
(c) \hspace*{1mm}\includegraphics{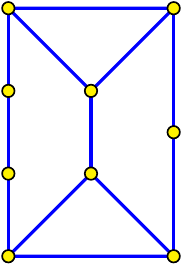}
\caption{Non-separable planar graphs: (a) outerplanar, (b) wheel, (c) elongated triangular prism.}
\end{figure}

\cref{PlanarNonSep} leads to a proof of the following result. 

\begin{lem}[{\protect\citep[Lemma 14]{HW26}}] 
\label{PlanarDetailed}
    For every bag $B$ of every refined tree-decomposition of a planar graph $G$, $G[B]$ is outerplanar, is a subgraph of a wheel, or is a subgraph of an elongated triangular prism.
\end{lem}

Each graph listed in \cref{PlanarDetailed} has treewidth at most $3$, so \cref{PlanarDetailed} implies the following result.

\begin{cor}[{\protect\citep[Corollary 15]{HW26}}] 
    Every planar graph $G$ has an optimal tree-decomposition $(B_x:x\in V(T))$ such that for each $x\in V(T)$, the subgraph $G[B_x]$ has treewidth at most $3$. 
\end{cor}

We now show that two of the three cases described in \cref{PlanarDetailed} correspond to graphs of pathwidth at most $3$.
\begin{lem}\label{lem:pw3graphs}
    If $H$ is a subgraph of a wheel or a subgraph of an elongated triangular prism, then $\pw(H)\leq 3$.
\end{lem}
\begin{proof}
    Since pathwidth is monotone with respect to the subgraph relation, it suffices to consider the case where $H$ is a wheel or an elongated triangular prism.

    If $H$ is a wheel with central vertex $v$, then for any $w\in V(H)\setminus \{v\}$ the graph $H-\{v,w\}$ is a path.
    Likewise, if $H$ is an elongated triangular prism obtained by subdividing $H':=K_3\square K_2$,  and $v$ and $w$ are non-adjacent vertices of $H'$, then $H-\{v,w\}$ is a path.

    In either case, we can obtain a path-decomposition for $H$ of width $3$ from an optimal path-decomposition of $H-\{v,w\}$ by adding $v$ and $w$ to every bag.
\end{proof}

In order to complete the proof of \cref{PlanarOptimalPW3bags}, it remains to exclude bags that induce outerplanar graphs of large pathwidth.
We achieve this in the next subsection by considering planar triangulations.

\subsection{Planar Triangulations}

A graph $G$ is a \defn{planar triangulation} if $G$ has an embedding in the plane in which each face is bounded by a 3-cycle. This section shows that planar triangulations have tree-decompositions in which each bag is 2-connected, in addition to the properties in \cref{PlanarDetailed}.  We use the following two folklore lemmas, where a \defn{minimal separator} of a connected graph $G$ is a set $S\subseteq V(G)$ such that $G-S$ is disconnected and $G-S'$ is connected for every proper subset $S'$ of $S$. 

\begin{lem}[{\protect\citep[Proposition~8.2.3]{MoharThom}}]\label{lem:separatingcycles}
Every minimal separator in a planar triangulation induces a cycle. 
\end{lem}

\begin{lem}[{\protect\citep[Lemma~2.3.3]{MoharThom}}]
\label{lem:3conn}
Every planar triangulation with at least four vertices is $3$-connected. 
\end{lem}

\begin{lem}\label{lem:minsep}
For every refined tree-decomposition $(B_x:x\in V(T))$ of a $3$-connected planar graph $G$, for each edge $uv\in E(T)$, the set $B_u\cap B_v$ is a minimal separator of $G$.
\end{lem}

\begin{proof}
    By \cref{ComponentNormalisation}, there is a unique component $C_u$ of $G-B_v$ that intersects $B_u$ and a unique component $C_v$ of $G-B_u$ that intersects $B_v$.
    By the definition of tree-decomposition, $C_u$ and $C_v$ are components of $G-(B_u\cap B_v)$.
    \begin{claim}
        Every vertex in $B_u\cap B_v$ has a neighbour in $C_u$ and in $C_v$. 
    \end{claim}
    \begin{proofclaim}
        Suppose for contradiction that some $w\in B_u\cap B_v$ has no neighbour in $C_u$.
        Let $T_1$ and $T_2$ be disjoint copies of $T$, let $X_1:=N[C_u]$ and $X_2:=V(G)\setminus V(C_u)$. For each $x\in V(T)$ and $i\in \{1,2\}$, let $(x,i)$ be the copy of $x$ in $T_i$ and define $B'_{(x,i)}:=B_x\cap X_i$.
        Let $T'$ be obtained from the disjoint union of $T_1$ and $T_2$ by adding the edge $(v,1)(v,2)$.
        Note that $G=G[X_1]\cup G[X_2]$ and $X_1\cap X_2=N(C_u)\subseteq B_v\setminus \{w\}$, so $\mathcal{D}':=(B'_{(x,i)}:(x,i)\in V(T'))$ is a tree-decomposition of $G$.
        By construction, $\mathcal{D}'$ is a refinement of $(B_x:x\in V(T))$, and since $w\in B_u\setminus X_1$ and $\emptyset\neq (B_u\cap V(C_u))\subseteq B_u\setminus X_2$, no bag of $\mathcal{D}'$ contains $B_u$, so it is a proper refinement.
        This contradicts the fact that $(B_x:x\in V(T))$ is refined.
        Thus every vertex in $B_u\cap B_v$ has a neighbour in $C_u$, and by symmetry every vertex in $B_u\cap B_v$ has a neighbour in $C_v$.
    \end{proofclaim}
    Suppose for contradiction that there is a third component $C'$ of $G-(B_u\cap B_v)$ distinct from $C_u$ and $C_v$.
    Since $G$ is $3$-connected, there are at least three vertices $v_1,v_2,v_3$ in $N(C')$. 
    Now by contracting each of $C_u$, $C_v$ and $C'$ to a single vertex and deleting all other vertices except $v_1$, $v_2$ and $v_3$ and all edges in $G[\{v_1,v_2,v_3\}]$, we find a $K_{3,3}$ minor in $G$, contradicting that $G$ is planar.
    Thus $C_u$ and $C_v$ are the only components of $G-(B_u\cap B_v)$, which together with the above claim implies that $B_u\cap B_v$ is a minimal separator in $G$. 
\end{proof}

\begin{lem}
\label{PlanarTriang}
For every refined tree-decomposition $(B_x:x\in V(T))$ of a planar triangulation $G$, for each $x\in V(T)$, the subgraph $G[B_x]$ is $2$-connected. 
\end{lem}
\begin{proof}
    The result is trivial if $G=K_3$. Now assume that $|V(G)|\geq 4$. By \cref{lem:3conn}, $G$ is $3$-connected, so we may assume $G\neq G[B_x]$, and so $x$ has some neighbour $x'$ in $T$.
    By \cref{lem:minsep}, $B_x\cap B_{x'}$ is a  separator, and so $|B_x|\geq |B_x\cap B_{x'}|\geq 3$.
    
    Suppose for contradiction that there is some $S\subseteq B_x$ with $|S|\leq 1$ such that $C_1$ and $C_2$ are distinct components of $G[B_x]-S$.
    Since $G$ is $3$-connected there is a path $P$ from $V(C_1)$ to $V(C_2)$ in $G-S$.
    Let $P'$ be a minimal subpath of $P$ with endpoints in distinct components of $G[B_x]-S$.
    By minimality, $P'$ has no internal vertex in $B_x$.
    In particular, $P'':=P'-(V(P')\cap B_x)$ is a non-empty connected subgraph of $G-B_x$.
    Thus there is a unique neighbour $w$ 
    of $x$ in $T$ such that $V(P'')\cap \bigcup\{B_z:z\in V(T_{w:x})\}$ is non-empty.
    Since each edge of $P'$ has an endpoint in $V(P'')\subseteq V(G_{w:x})$, we have $V(P')\subseteq \bigcup\{B_z:z\in V(T_{w:x})\}$.
    In particular, both endpoints of $P'$ are in $B_x\cap B_w$, and so $G[B_x\cap B_w]-S$ has at least two components.
    But by \cref{lem:separatingcycles,lem:minsep}, $G[B_x\cap B_w]$ is $2$-connected.
    This contradiction completes the proof.
\end{proof}

We now prove our main technical result for planar triangulations. 

\begin{thm}\label{thm:TriangDetailed}
    Let $G$ be a planar triangulation with $|V(G)|\geq 4$ and let $(B_x:x\in V(T))$ be a refined tree-decomposition of $G$. For each $x\in V(T)$ the graph $G[B_x]$ is isomorphic to one of the following:
    \begin{enumerate}[a)]
        \item a graph $H$ with a vertex $v$ of degree at least $3$ such that $H-v$ is a cycle,
        \item a subdivision of $K_{2,3}$,
        \item a subdivision of $K_4$ that contains a triangle, or
        \item an elongated triangular prism.
    \end{enumerate}
    In particular, $\pw(G[B_x])\leq 3$.
\end{thm}

\begin{figure}[!ht]
(a) \hspace*{-8mm} \includegraphics{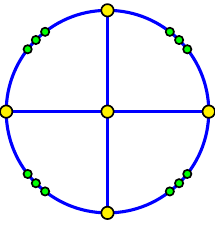}
\hspace*{-2mm}
(b) \hspace*{-5mm} \includegraphics{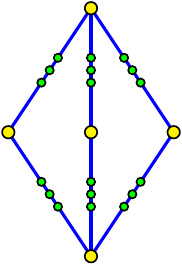}
\hspace*{-2mm}
(c) \hspace*{-6mm} \includegraphics{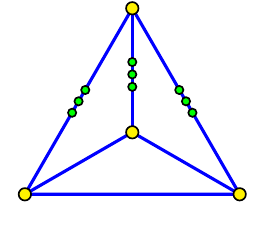}
\hspace*{-6mm}
(d) \hspace*{1mm}\includegraphics{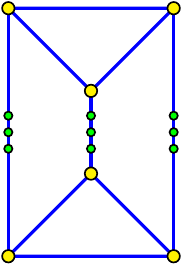}
\caption{Graphs in \cref{thm:TriangDetailed} (where green vertices are a possibly empty sequence of subdivision vertices).}\label{fig:characteriation}
\end{figure}

\begin{proof}
    By \cref{PlanarTriang}, $G':=G[B_x]$ is $2$-connected. 
    Suppose for contradiction that $G'$ is outerplanar. 
    If $G'$ has a chord $vw$, then there are vertices $v'$ and $w'$ in distinct components of $G'-\{v,w\}$. It follows that there is no induced cycle containing $v'$ and $w'$.
    By \cref{lem:separatingcycles,lem:3conn,lem:minsep}, there is no pair of adjacent vertices $x_1,x_2$ of $T$ such that $\{v',w'\}\subseteq B_{x_1}\cap B_{x_2}$.
    This contradicts \cref{cor:adhesionsets}.
    Thus $G[B_x]$ is a $2$-connected outerplanar graph with no chord; that is, a cycle. Since $(B_x:x\in V(T))$ is normal, there is no neighbour $x'$ of $x$ in $T$ with $B_x\cap B_{x'}=B_x$, and so by \cref{lem:separatingcycles,lem:3conn,lem:minsep} there is no neighbour of $x$ in $T$. Thus $V(T)=
    \{x\}$. Hence $V(G)=B_x$, and by \cref{cor:adhesionsets} there is no pair of non-adjacent vertices in $G$.
    It follows that $G$ is a $3$-vertex cycle, contradicting that $|V(G)|\geq 4$.
    Thus $G'$ is not outerplanar.
    
    Now by \cref{PlanarDetailed}, $G[B_x]$ is a non-outerplanar subgraph of either a wheel or an elongated triangular prism, and thus has pathwidth at most $3$ by \cref{lem:pw3graphs}.
    It is routine to verify that every non-outerplanar, $2$-connected subgraph of a wheel can be reduced to a cycle by deleting its central vertex, and so is covered by either case (a) or (b).
    Likewise, every proper subgraph of an elongated triangular prism that is non-outerplanar is covered by either case (b) or (c), and the elongated triangular prism itself is covered by case (d).
   \end{proof}

We remark that each outcome in \cref{thm:TriangDetailed} is necessary, in the sense that a graph $H$ is induced by a bag of some refined tree-decomposition of some planar triangulation $G$ if and only if $H$ is described by one of the four cases of \cref{thm:TriangDetailed}. 
To see this, let $H$ be such a graph and let $G$ be the plane triangulation obtained from $H$ by adding, for each non-clique face $F$ of $H$, a new vertex $v_F$ embedded in $F$ and adjacent to all vertices incident to $F$. Then $G$ has a tree-decomposition, indexed by a star with central vertex $x$ and a leaf $x_F$ for each face $F$ of $H$, such that $B_x:=V(H)$ and $B_{x_F}:=N_G[v_F]$ for each face $F$ of $H$.
Note that $G$ is isomorphic to $G\llbracket B_x\rrbracket$, $G$ is a planar triangulation and $G[B_x]= H$.
In each of the four cases, it is not difficult to verify that for each bag $B$ of this tree-decomposition and every pair $v,w$ of non-adjacent vertices in $B$, there is a path with endpoints $v$ and $w$ that is internally disjoint from $B$. Thus by \cref{thm:refinedcharacterisation}, this is a refined tree-decomposition of $G$.

\subsection{General Planar Graphs}

To extend \cref{thm:TriangDetailed} to general planar graphs, we call upon the following lemma of \citet{BV13}.

   \begin{lem}[{{\protect\citep[Lemma~1]{BV13}}}]
   \label{lem:twtriangulation}
       For every planar graph $G$ with $|V(G)|\geq 4$, there is a planar triangulation $G'$ such that $G$ is a spanning subgraph of $G'$ and $\tw(G')=\max\{3,\tw(G)\}$.
   \end{lem}

We now prove the primary theorem from the introduction. 

\PlanarOptimalPW*

\begin{proof}
    If $G$ has a tree-decomposition with width at most 3, then every bag has at most four vertices and thus has pathwidth at most 3. Hence we may assume $\tw(G) \geq 4$. By \cref{lem:twtriangulation} there is a planar triangulation $G'$ such that $G$ is a spanning subgraph of $G'$ and $\tw(G')=\tw(G)$.
    Let $\DD$ be any refined optimal tree-decomposition of $G'$. Since $\tw(G')=\tw(G)$ and $G\subseteq G'$, $\DD$ is also an optimal tree-decomposition of $G$. By  \cref{thm:TriangDetailed}, each bag of $\DD$ has pathwidth at most $3$.
\end{proof}

We now prove a version of \cref{thm:TriangDetailed} for arbitrary planar graphs.

\begin{thm}\label{thm:planarfulldetail}
    Given a refined tree-decomposition $\mathcal{D}=(B_x:x\in V(T))$ of a planar graph $G$, there is a planar graph $G'$ such that $G$ is a spanning subgraph of $G'$, $\mathcal{D}$ is a (refined) tree-decomposition of $G'$, and for each bag $B\in \mathcal{D}$ the graph $G'[B]$ is isomorphic to one of the following:
    \begin{enumerate}[a)]
        \item a graph $H$ with a vertex $v$ of degree at least $3$ such that $H-v$ is a cycle,
        \item a subdivision of $K_{2,3}$,
        \item a subdivision of $K_4$ that contains a triangle,
        \item an elongated triangular prism, or
        \item a clique with at most three vertices.
    \end{enumerate}
    In particular, $\pw(G[B])\leq 3$.
\end{thm}
\begin{proof}
    Let $G'$ be a supergraph of $G$ that is edge-maximal subject to the conditions that $V(G')=V(G)$, $G'$ is planar and $\mathcal{D}$ is a tree-decomposition of $G'$.
    Thus for any bag $B$ of $\mathcal{D}$ and any pair $v,w$ of non-adjacent vertices in $G'[B]$, the graph $G'+vw$ is non-planar.
    Let $E^*$ be the set of edges $xy\in E(T)$ such that $|B_x\cap B_y|\leq 2$, and let $\mathcal{C}$ be the set of components of $T-E^*$.
    For each $C\in \mathcal{C}$, let $G'_C:=G'[\bigcup\{B_x:x\in V(C)\}]$, and note that $\mathcal{D}_C:=(B_x:x\in V(C))$ is a tree-decomposition of $G'_C$.

    \begin{clms}\label{clm1}
        If $xy\in E^*$ and $B_x\cap B_y=\{v,w\}$, then $vw\in E(G')$.
    \end{clms}
    \begin{proofclaim}
    Suppose for contradiction that $vw\notin E(G')$, and consider the graphs $G'_1:=G'[\bigcup \{B_z:z\in V(T_{x:y})\}]$ and $G'_2:=G'[\bigcup \{B_z:z\in V(T_{y:x})\}]$.
    Since $\mathcal{D}$ is normal, there is a vertex $v'$ in $B_x\setminus B_y$. 
    By \cref{thm:refinedcharacterisation}, $G$ contains a path $Q_1$ from $v$ to $v'$ that is internally disjoint from $B_x$ and a path $Q_2$ from $v'$ to $w$ that is internally disjoint from $B_x$. Thus there is a path $Q$ from $v$ to $w$ in $Q_1\cup Q_2$, which by construction is a path in $G'_1$.
    By symmetry, there is also a path $P$ from $v$ to $w$ in $G'_2$.
    By assumption, $P$ has at least one internal vertex $u$.
    Now in the embedding of $G'_1$, the boundary of the face containing $u$ contains both $v$ and $w$, and $v$ and $w$ are in the same component of the boundary of this face since $v$ and $w$ are in the same component of $G'_2$.
    It follows that we can draw a new edge from $v$ to $w$ without breaking planarity, contradicting the edge maximality of $G'$. 
    \end{proofclaim}
    
    \begin{clms}\label{clm2}For every $C \in \mathcal{C}$, $\mathcal{D}_C$ is a refined tree-decomposition of $G'_C$.
    \end{clms}
    \begin{proofclaim}
    First observe that since $G'$ is a supergraph of $G$ and $\mathcal{D}$ is a refined tree-decomposition of $G$ and a tree-decomposition of $G'$, $\mathcal{D}$ is a refined tree-decomposition of $G'$.
    Consider an arbitrary bag $B$ of $\mathcal{D}_C$, and an arbitrary pair of vertices $v$ and $w$ in $B$. 
    By \cref{thm:refinedcharacterisation}, there is a path $P$ from $v$ to $w$ in $G'$ with no internal vertex in $B$.
    We may choose this path to be induced.
    Suppose for contradiction that $P$ is not a path in $G'_C$.
    Then there is an edge $xy\in E^*$ with $x$ in $C$ and $y$ not in $C$ such that $P$ contains a vertex $u$ in $\bigcup\{B_z\setminus B_y:z\in V(T_{y:x})\}$.
    Let $P_1$ be the subpath of $P$ from $v$ to $u$ and let $P_2$ be the subpath of $P$ from $u$ to $w$, and note that both $P_1$ and $P_2$ contain a vertex in $B_x\cap B_y$. 
    But now by \cref{clm1}, there is an edge from a vertex in $P_1-u$ to a vertex in $P_2-u$, contradicting that $P$ is induced.
    Thus there is a path in $G'_C$ from $v$ to $w$ that is internally disjoint from $B$.
    Now by \cref{thm:refinedcharacterisation}, $\mathcal{D}_C$ is a refined tree-decomposition of $G'_C$.
    \end{proofclaim}
    
    \begin{clms}\label{clm3}For every $C \in \mathcal{C}$, $G'_C$ is edge maximal subject to being a planar supergraph of $G[\bigcup\{B_x:x\in V(C)\}]$ with $\mathcal{D}_C$ as a tree-decomposition.
    \end{clms}
    \begin{proofclaim}
    For each $C\in \mathcal{C}$, let $G^*_{C}$ be a planar supergraph of $G'_{C}$ that is edge-maximal with respect to having $\mathcal{D}_C$ as a tree-decomposition.
    Then the graph $G^*:=\bigcup \{G^*_{C}:C\in \mathcal{C}\}$ can be constructed by iteratively pasting planar graphs together on edges. This operation (often called a $2$-sum) preserves planarity, and so $G^*$ is planar.
    By construction, $G^*$ is a planar supergraph of $G'$, and $\mathcal{D}$ is a tree-decomposition of $G^*$ since every edge of $G^*$ is contained in a bag of $\mathcal{D}$.
    Since $G'$ is edge maximal with this property, $G'=G^*$.
    Thus $G^*_C=G'_C$, which proves the claim. 
    \end{proofclaim}
    
    Fix $C \in \mathcal{C}$.
    If $|V(G'_C)|\leq 3$, then we must have $|V(C)|=1$, since otherwise there is an edge $xy\in E(C)$ and either $B_x=B_y$, which contradicts the assumption that $\mathcal{D}$ is refined, or $xy\in E^*$, which contradicts the definition of $C$.
    Thus there is unique vertex $x$ in $C$.
    Now \cref{clm2} together with \cref{thm:refinedcharacterisation} imply that $G'[B_x]=G'_C[B_x]$ is a clique.

    Now suppose that $|V(G'_C)|\geq 4$, let $(A,B)$ be a non-trivial separation of $G_C'$, and suppose for contradiction that $|A\cap B|\leq 1$.
    Since some bag of $\mathcal{D}_C$ contains $A\cap B$, \cref{thm:refinedcharacterisation} implies that there is no $x\in V(C)$ such that both $A\cap B_x$ and $B\cap B_x$ are non-empty.
    For each $x\in V(C)$, colour $x$ red if $A\cap B_x=\emptyset$, and blue otherwise.
    Since $(A,B)$ is non-trivial, there is at least one red and one blue vertex in $C$, and since $C$ is a tree there is some red vertex $x$ adjacent to some  blue vertex $y$. Then $B_x\cap B_y\subseteq A\cap B$, and so $xy\in E^*$, which contradicts the definition of $C$. Hence, $G'_C$ is $2$-connected.
    Fix an embedding $\phi$ of $G'_C$ in the plane, and note that since $G'_C$ is $2$-connected, the boundary of every face is a cycle. A cycle that forms the boundary of a face is \defn{facial cycle}.

    Suppose for contradiction that there is a face $F$ of $\phi$ whose facial cycle $O$ is not a triangle.
    Note that the restriction of $\mathcal{D}$ to $V(O)$ is a tree-decomposition of $O$, so some bag $B_x$ of $\mathcal{D}$ contains three vertices of $O$ since $O$ has treewidth $2$.
    Since $O$ is not a triangle, there are two vertices $v$ and $w$ in $B_x\cap V(O)$ which are not adjacent in $O$.
    Note that $vw\in E(G'_C)$, or else we could add $vw$ to $G'_C$ and draw this edge in $F$, thus maintaining planarity and contradicting \cref{clm3}.
    Note that $vw$ is drawn outside of $F$ in $\phi$ since $F$ is a face.
    Let $P$ and $Q$ be the two components of $O-\{v,w\}$, and suppose for contradiction that there is a path $P'$ in $G'_C-\{v,w\}$ from $V(P)$ to $V(Q)$. 
    Again, since $F$ is a face, $P'$ is drawn outside of $F$ in $\phi$.
    By contracting $P$ and $Q$ down to single vertices and contracting $P'$ down to an edge, we obtain a planar embedding of $K_4$ with all vertices incident to $F$. In other words, we obtained an outerplanar embedding of $K_4$, which is impossible.
    Thus $P$ and $Q$ are in separate components $X_P$ and $X_Q$ of $G'_C-\{v,w\}$.
    Let $A_1:=V(X_P)$ and $A_2:=V(G'_C)\setminus (A_1\cup \{v,w\})$.
    Since $\{v,w\}\subseteq B_x$ and by \cref{clm2}, there is no bag $B_y$ of $\mathcal{D}_C$ that intersects both $A_1$ and $A_2$.
    Colour each $y\in V(C)$ red if $B_y$ intersects $A_1$ and blue otherwise. There must be an adjacent red and blue vertex $x_1$ and $x_2$ in $C$, but $B_{x_1}\cap B_{x_2}\subseteq \{v,w\}$, so $x_1x_2\in E^*$ contradicting the definition of $C$.
    Thus $G'_C$ is a planar triangulation.
    Thus by \cref{thm:TriangDetailed}, for every $x\in V(C)$, the graph $G'[B_x]=G'_C[B_x]$ satisfies the theorem hypothesis.  
\end{proof}

\subsection{Many Bags}
\label{ManyBags}

This section proves \cref{PlanarOptimalManyBags}, showing that every planar graph has an optimal tree-decomposition in which the union of any set of $k$ bags has pathwidth $O(k)$. The following language will be helpful. Let $C_1$ and $C_2$ be distinct induced cycles in a plane graph $G$. We say $C_1$ is \defn{inside} $C_2$ if some edge of $C_1$ is in the interior of $C_2$, and $C_1$ is \defn{outside} $C_2$ if some edge of $C_1$ is in the exterior of $C_2$. We say $C_1$ \defn{crosses} $C_2$ if $C_1$ is both inside and outside $C_2$. If $C_1$ crosses $C_2$, then $C_2$ crosses $C_1$, so we just say that $C_1$ and $C_2$ \defn{cross}. Cycles $C_1,\dots,C_k$ in a planar graph $G$ are \defn{non-crossing} if $C_i$ and $C_j$ do not cross for all distinct $i,j\in\{1,\dots,k\}$. 

We use the following result of \citet{FLW-JGAA03}.
\begin{thm}[{\protect\cite[Theorem 1]{FLW-JGAA03}}]
\label{thm:geometricpw}
    Let $G$ be a plane graph with straight edges such that there is a set of at most $k$ parallel lines in the plane hitting the vertices of $G$. Then $\pw(G)\leq k$.
\end{thm}

\begin{lem}
\label{PairwiseDisjointCycles}
    Let $C_1,\dots,C_k$ be pairwise disjoint induced cycles in a planar graph $G$ with $k\geq 2$. Then $\pw(G[V(C_1)\cup\dots\cup V(C_k)])\leq 14k-24$.
\end{lem}
\begin{proof}
      We may assume that $G$ is connected and $V(G)=V(C_1)\cup\dots\cup V(C_k)$. Fix a plane embedding of $G$. Let $\mathcal{F}$ be the set of faces of the spanning subgraph $H:=C_1\cup \cdots \cup C_k$ of $G$.
    For each $F\in \mathcal{F}$, let $\mathcal{C}_F$ be the set of cycles of $H$ incident to $F$.
    Let $G_F$ be the plane multigraph obtained from $G$ by deleting all vertices not in $\bigcup \{V(C):C\in \mathcal{C}_F\}$, deleting one edge of each cycle in $\mathcal{C}_F$, and then contracting the paths induced by the sets $V(C)$ with $C\in \mathcal{C}_F$ into vertices (preserving parallel edges).
    Let $E_{F,1}$ be the set of edges of $G_F$ that are incident to at least one face of $G_F$ of order at least $3$, let $E_{F,2}$ be the set of edges that are only incident to faces of $G_F$ of order $2$ and let $G_F'=G_F-E_{F,2}$. 
    If $|\mathcal{C}_F|\leq 2$, then observe that $|E_{F,1}|=0$.

    Suppose that $|\mathcal{C}_F|\geq 3$.
    By the construction, no two faces of order $2$ share an edge in $G_F'$. 
    In particular, removing an edge incident to a face of order $2$ in $G_F'$ merges this face with another face creating a face of order at least $3$.
    For each face of order $2$ in $G_F'$, we choose an incident edge and call the set of these edges $E_F'$.
    Note that $|E_F'| \leq |E(G_F')\setminus E_F'|$.
    Let $G_F'' = G_F' - E_F'$.
    Since $|V(G_F'')| = |V(G_F')| = |V(G_F)| = |\mathcal{C}_F|$, $G_F''$ has at least three vertices and we may apply Euler's formula to obtain $|E(G_F'')|\leq 3|\mathcal{C}_F|-6$.
    Altogether,
    \begin{align*}
        |E_{F,1}|=|E(G_F')| = |E_F'| + |E(G_F')\setminus E_F'| \leq 2|E(G_F'')|\leq 6|\mathcal{C}_F|-12.
    \end{align*}
    
    Let $S_F\subseteq V(G)$ be the set of all endpoints of edges of $G$ corresponding to edges in $E_{F,1}$.
    Now consider an arbitrary cycle $C\in \mathcal{C}_F$, and let $e_1, \dots e_t$ be the edges of $G$ that are incident to $C$ and embedded in $F$, in clockwise order around $C$, and for convenience set $e_{t+1}=e_1$.
    If $i\in \{1,\dots ,t\}$ is such that $e_i$ and $e_{i+1}$ are incident to distinct cycles in $\mathcal{C}_F\setminus \{C\}$, then there is a face of order at least $3$ in $G_F$ incident to both $e_i$ and $e_{i+1}$, and so both $e_i$ and $e_{i+1}$ are in $E_{F,1}$.

    Let $S:=\bigcup \{S_F:F\in \mathcal{F}\}$.
    For each positive integer $i$ let $\mathcal{F}_i$ be the set of faces $F\in \mathcal{F}$ such that $|\mathcal{C}_F|=i$, and let $\mathcal{F}_{\geq 3}=\mathcal{F}\setminus (\mathcal{F}_1\cup \mathcal{F}_2)$.
    Let $a:=|\mathcal{F}_1|$ and $b:=|\mathcal{F}_2|$.
    Since every cycle of $H$ is incident with exactly two faces in $\mathcal{F}$, $\sum_{F\in \mathcal{F}}|\mathcal{C}_F|=2k$.
    Note also that $|\mathcal{F}|=k+1$.
    Now
    \begin{align*}    
    |S|\leq 2\sum_{F\in \mathcal{F}}|E_{F,1}|
     \leq 2\sum_{F\in \mathcal{F}_{\geq 3}} \!\!(6|\mathcal{C}_F|-12)
    &\leq \Big(12\sum_{F\in \mathcal{F}_{\geq 3}}\!\!|\mathcal{C}_F|\Big)
    -24|\mathcal{F}_{\geq 3}|\\
    &=12(2k-a-2b)-24(k+1-a-b)\\
    & =12a-24.
    \end{align*}
    In particular, $a=|\mathcal{F}_1|\leq k$ since $k\geq 2$, so $|S|\leq 12k-24$.

    For each cycle $C$ of $H$ that is disjoint from $S$ choose a vertex in $C$, and let $S'$ be the set of these vertices.
    Let $G':=G-(S\cup S')$.
    Let $\mathcal{P}$ be the set of components of $H-(S\cup S')$ (which by construction is a set of paths).
    By the above paragraph, in $G'$ each $P\in \mathcal{P}$ has neighbours in at most two other paths in $\mathcal{P}$, at most one for each $F\in \mathcal{F}$ incident to $P$.
    In particular, if $P\in \mathcal{P}$ has neighbours in two other paths in $\mathcal{P}$, then one of these is contained in a cycle of $H$ that is embedded inside the cycle of $H$ containing $P$.
    From this we may deduce that for any component $X$ of $G'$, the paths $P_1,\dots, P_t\in \mathcal{P}$ contained in $X$ come from distinct cycles in $H$ and can be ordered so that each edge in $E(X)\setminus E(H)$ lies between $P_i$ and $P_{i+1}$ for some $i\in \{1,\dots, t\}$.
    Note that for each $i\in \{1,\dots, t\}$ the edges between $P_i$ and $P_{i+1}$ are all drawn in the same face of $H$.
    It follows that $X$ admits a straight line drawing such that the paths $P_1,\dots,P_t$ are drawn in $t$ distinct parallel lines in the plane.
    Thus by \cref{thm:geometricpw}, the pathwidth of $X$ is at most $t$, which is at most $k$.
    Thus $\pw(G)\leq \pw(G')+|S\cup S'|\leq k+(12k-24)+k\leq 14k-24$.
    \end{proof}

\begin{lem}\label{lem:noncrossingpw}
    Let $C_1,\dots,C_k$ be non-crossing induced cycles in a plane graph $G$ with $V(G)=V(C_1)\cup\dots \cup V(C_k)$ and $k\geq 2$. Then $\pw(G) \leq 14k-24$.
\end{lem}

\begin{proof}
    Let $G'$ be a plane graph with the following properties:
    \begin{enumerate}
        \item\label{prop1}$G\preccurlyeq G'$,
        \item\label{prop2} there exists a set $\{C'_1,\dots , C'_k\}$ of $k$ pairwise non-crossing induced cycles in $G'$ with $V(G')=V(C'_1)\cup\cdots \cup V(C'_k)$, and
        \item\label{prop3}subject to \ref{prop1} and \ref{prop2}, the number of pairs of cycles in $\{C'_1,\dots,C'_k\}$ that share a vertex is minimised.
    \end{enumerate}
    Such a graph exists since $G'=G$ is a candidate. Suppose for contradiction that the cycles in $\{C'_1,\dots,C'_k\}$ are not pairwise disjoint. Without loss of generality, $C'_1$ and $C'_2$ share a vertex and $C'_2$ is outside $C'_1$. Let $I$ be the set of integers $i\in \{1,\dots, k\}$ such that $C'_i$ is inside $C'_1$ and let $O$ be the set of integers $i\in \{1,\dots,k\}$ such that $C'_i$ is outside $C'_1$.
    Let $G_1:=G'[V(C'_1)\cup \bigcup \{V(C'_i):i\in I\}]$, and let $G_2:=G'[V(C'_1)\cup \bigcup \{V(C'_i):i\in O\}]$.
    Let $G_3$ be obtained from the disjoint union of a copy $H_1$ of $G_1$ and a copy $H_2$ of $G_2$ by adding an edge for every $v\in V(C'_1)$ between the copy of $v$ in $H_1$ and the copy of $v$ in $H_2$.
    A planar drawing of $G_3$ can easily be obtained by drawing $H_2$ according to the induced drawing of $G_2$, and then drawing $H_1$ inside $C'_1$ according to a scaled down version of the induced drawing of $G_1$.
    Let $E^*$ be the set of edges from $H_1$ to $H_2$ in $G_3$ whose endpoints are copies of a vertex in $V(C'_1)$ which is not in any cycle in $\{C'_i:i\in O\}$.
    Let $G^*$ be obtained from $G_3$ by contracting all edges in $E^*$.
    For each $i\in O$, let $C^*_i$ be the cycle in $G^*$ corresponding to the copy of $C'_i$ in $H_2$.
    For each $i\in \{1\}\cup I$, let $C^*_i$ be the cycle of $G^*$ corresponding to the copy of $C'_i$ in $H_1$.
    Now since $C'_2$ is outside $C'_1$, no vertex of the copy of $C'_2$ in $H_2$ is incident with an edge in $E^*$, and so $V(C^*_1)\cap V(C^*_2)=\emptyset$.
    Since $G_3$ is planar, $G^*$ is planar.
    The graph obtained from $G^*$ by contracting all edges with exactly one endpoint in $V(H_1)\setminus V(H_2)$ and one endpoint in $V(H_2)\setminus V(H_1)$ is isomorphic to $G'$, so $G\preccurlyeq G'\preccurlyeq G^*$.
    By construction, $V(G^*)=V(C^*_1)\cup \cdots \cup V(C_k^*)$.
    However, for every pair of integers $i,j\in \{1,\dots,k\}$ such that $V(C^*_i) \cap V(C^*_j)\neq \emptyset$, we have $V(C'_i)\cap V(C'_j)\neq \emptyset$.
    Since $V(C^*_1)\cap V(C^*_2)=\emptyset$, this contradicts the fact that $G'$ satisfies \ref{prop3}.
    Thus, the cycles in $\{C'_1,\dots,C'_k\}$ are pairwise vertex disjoint.

    Now $\pw(G')\leq 14k-24$ by \cref{PairwiseDisjointCycles}, and so  $\pw(G)\leq 14k-24$ by \ref{prop1}.
\end{proof}

\begin{thm}\label{thm:manybagsTriang}
    Let $G$ be a planar triangulation, let $\mathcal{D}:=(B_x:x\in V(T))$ be a refined tree-decomposition of $G$ and let $X$ be a set of $k$ distinct vertices of $T$, with $k\geq 2$.
    The graph $G[\bigcup\{B_x:x\in X\}]$ has pathwidth at most $28k-24$.
\end{thm}
\begin{proof}
    Given a bag $B$ of $\mathcal{D}$ such that $G[B]$ satisfies condition b, c or d of \cref{thm:TriangDetailed}, it is easy to find a pair of non-crossing cycles $C$ and $C'$ such that $B=V(C)\cup V(C')$.
    Thus, by \cref{thm:TriangDetailed}, for each $x\in X$ there are sets $B'_x$ and $S_x$ such that $B_x=B'_x\cup S_x$, $G[B'_x]$ is a cycle and either $|S_x|=1$ or $G[S_x]$ is a cycle that does not cross $G[B'_x]$.

    Let $X_1$ be the set of vertices in $X$ such that $|S_x|=1$, and let $X_2:=X\setminus X_1$.
    If two cycles $C$ and $C'$ in $G$ cross, then $V(C')$ intersects at least two components of $G-V(C)$.
    Thus by \cref{ComponentNormalisation}, no pair of cycles in $\{G[B'_x]:x\in X\}\cup \{G[S_x]:x\in X_2\}$ cross.
    Let $Z_1:=\bigcup \{B'_x:x\in X\}$ and let $Z_2:=\bigcup \{S_x:x\in X_2\}$.
    By \cref{lem:noncrossingpw}, the pathwidth of $G[Z_1\cup Z_2]$ is at most $14(k+|X_2|)-24$.
    Finally, the pathwidth of $G[\bigcup\{B_{x}:x\in X\}]$ is at most $14(k+|X_2|)-24+|\bigcup\{S_x:x\in X_1\}|=14k-24+14|X_2|+|X_1|\leq 28k-24$.
\end{proof}

We now prove the following result from the introduction. 

\PlanarOptimalManyBagsStatement*

\begin{proof}
    If $\tw(G)\leq 3$ then every optimal tree-decomposition has this property since $\pw(G[\bigcup\{B_x:x\in X\}])\leq |\bigcup\{B_x:x\in X\}|\leq 4|X|\leq 28|X|-24$.
    Now assume $|V(G)|\geq 4$ and $\tw(G)\geq 4$. By \cref{lem:twtriangulation} there is a planar triangulation $G'$ such that $G$ is a spanning subgraph of $G'$ and $\tw(G')=\tw(G)$.
    By \cref{thm:manybagsTriang} it suffices to take a refined optimal tree-decomposition of $G'$.
\end{proof}

\subsection{Linear Grid Minors}

\citet{RST94} showed that for planar graphs, the bound in the Grid Minor Theorem can be improved to linear, in particular, every planar graph with treewidth at least $6k-4$ contains a $k\times k$-grid minor (see \citep{Grigoriev11} for slightly improved bounds). As an aside, we now demonstrate how our machinery can be used to obtain a short proof of this result with slightly worse bounds. 

Our proof relies on the following result that is implicitly used in (6.3) by \citet{RST94}, who say the proof is a `straightforward modification' of the proof of (4.1) by \citet{RS-III}; see \citep[Lemma~7]{DMNW} for a complete proof.

\begin{lem}
\label{FindGridMinor}
Let $C=(v_1,\ldots,v_{4t})$ be the cycle bounding the outerface in a plane graph $G$, for some integer $t\geq 2$. 
Let $A'=\{v_1,\ldots,v_t\}$, $A=\{v_{t+1},\ldots,v_{2t}\}$, $B'=\{v_{2t+1},\ldots,v_{3t}\}$, and $B=\{v_{3t+1}, \ldots, v_{4t}\}$. If $G$ contains $t$ pairwise vertex disjoint $(A,B)$-paths and $t$ pairwise vertex disjoint $(A',B')$-paths then $G$ contains a $t \times t$ grid minor.
\end{lem}

\begin{cor}
\label{FindGridMinorWithoutCycle}
Let $C$ be the cycle bounding the outerface in a plane graph $G$. Let $v_1,\ldots,v_{4t}$ be vertices in $C$ in order, let $A'=\{v_1,\ldots,v_t\}$, $A=\{v_{t+1},\ldots,v_{2t}\}$, $B'=\{v_{2t+1},\ldots,v_{3t}\}$, and $B=\{v_{3t+1}, \ldots, v_{4t}\}$. If $G$ contains $t$ pairwise vertex disjoint $(A,B)$-paths and $t$ pairwise vertex disjoint $(A',B')$-paths then $G-E(C)$ contains a $(t-2) \times (t-2)$ grid minor.
\end{cor}

A set $S$ of vertices in a graph $G$ is \defn{reducible} if there is a separation $(A,B)$ of $G$ such that both $|(S\cup B)\cap A|$ and $|(S\cup A)\cap B|$ are strictly less than $|S|$, otherwise $S$ is \defn{irreducible}. 

\begin{lem}[{\protect\citep[Lemma 12]{HW26}}]
\label{AtomicIrreducible}
Every bag of an atomic tree-decomposition of a graph is irreducible.
\end{lem}

\begin{thm}
    For every positive integer $k$ and every planar graph $G$ of treewidth at least $15k+21$, the $k\times k$-grid is a minor of $G$.
\end{thm}
\begin{proof}
    Consider an atomic tree-decomposition $\mathcal{D}=(B_x:x\in V(T))$ of $G$, and let $G'$ be the planar supergraph guaranteed by \cref{thm:planarfulldetail} (recall that an atomic tree-decomposition is refined, see \Cref{lem:refinement}).
    Let $B$ be a bag of $\mathcal{D}$ of size at least $15k+22$.
    Since $15k+22\geq 4$, $H:=G'[B]$ satisfies condition (a), (b), (c) or (d) in \cref{thm:planarfulldetail}.
    If $H$ satisfies (a), let $u$ and $w$ be neighbours of $v$ in the cycle $H-v$. Let $P_1$ and $P_2$ be the two internally disjoint paths from $u$ to $w$ in $H-v$, and let $P_3$ be the path $uvw$.
    Let $O_1:=P_2\cup P_3$, $O_2:=P_1\cup P_3$ and $O_3:=P_1\cup P_2$.

    If $H$ satisfies (b), (c) or (d), consider the planar embedding of $H$ in \cref{fig:characteriation}, and let $O_1$, $O_2$ and $O_3$ be distinct facial cycles of $H$ such that all other facial cycles are triangles.

    In all cases, it can be quickly verified that each vertex of $H$ is in at least two of the cycles in $\{O_1,O_2,O_3\}$.
    Thus there is a cycle $O\in \{O_1,O_2,O_3\}$ with $|V(O)|\geq \lceil\frac{2}{3}|B|\rceil\geq  10k+15$.
    Let $Q_1$, $Q_2$, $Q_3$, $Q_4$ and $Q_5$ be disjoint paths in $O$ with $2k+3$ vertices each, listed in cyclic order around $O$.

    For distinct $i,j\in \{1,2,3,4,5\}$, let $B^-_{i,j}:=B\setminus (V(Q_i)\cup V(Q_j))$ and let $S_{i,j}$ be a minimum set of vertices separating $V(Q_i)$ from $V(Q_j)$ in $G-B^-_{i,j}$. By \cref{AtomicIrreducible}, $B$ is irreducible in $G$, and so $|S_{i,j}|\geq \min\{|V(Q_i)|,|V(Q_j)|\}=2k+3$.
    By Menger's Theorem, there is a set $\mathcal{P}_{i,j}$ of $2k+3$ pairwise vertex disjoint paths from $V(Q_i)$ to $V(Q_j)$ in $G-B^-_{i,j}$.
    Let $\mathcal{P}_{i,j}^{\textrm{in}}$ be the set of paths in $\mathcal{P}_{i,j}$ that are embedded inside $O$ (discounting their endpoints) and let $\mathcal{P}_{i,j}^{\textrm{out}}$ be the set of paths in $\mathcal{P}_{i,j}$ that are embedded outside $O$ (discounting their endpoints).
    Since each path in $\mathcal{P}_{i,j}$ is internally disjoint from $O$, $\mathcal{P}_{i,j}=\mathcal{P}_{i,j}^{\textrm{in}}\cup \mathcal{P}_{i,j}^{\textrm{out}}$.
    
    Let $\mathcal{X}:=\{(1,3),(1,4),(2,4),(2,5),(3,5)\}$.
    Fix a planar embedding of $G$, and for each $(i,j)\in \mathcal{X}$, colour $(i,j)$ red if $|\mathcal{P}_{i,j}^{\textrm{in}}|\geq k+2$ and colour $(i,j)$ blue if $|\mathcal{P}_{i,j}^{\textrm{out}}|\geq k+2$.
    Since $|\mathcal{P}_{i,j}|= 2k+3$, each $(i,j)\in \mathcal{X}$ is assigned a colour this way.
    Without loss of generality, there are at least three red pairs in $\mathcal{X}$. Thus there are distinct $(i,j)$ and $(i',j')$ in $\mathcal{X}$ with $i,j,i',j'$ all distinct such that $|\mathcal{P}_{i,j}^{\textrm{in}}|\geq k+2$ and $|\mathcal{P}_{i',j'}^{\textrm{in}}|\geq k+2$. Since $(i,j)$ and $(i',j')$ are distinct elements of $\mathcal{X}$ with $i,j,i',j'$ all distinct, without loss of generality $(Q_i,Q_{i'},Q_j,Q_{j'})$ is the cyclic order of these paths around $O$.
    Consider the graph $G'':= \bigcup (O\cup\mathcal{P}_{i,j}^{\textrm{in}}\cup \mathcal{P}_{i',j'}^{\textrm{in}})$.
    By \cref{FindGridMinorWithoutCycle}, the graph $G''-E(O)$ contains a $k\times k$ grid minor. 
    Thus $G$ contains a $k\times k$ grid minor, since $G''-E(O)$ is a subgraph of $G$.
\end{proof}

\section{Extensions}
\label{Extensions}

This section considers a broad collection of minor-monotone graph parameters $\beta$, which we call self-similar and robust. This includes treedepth, pathwidth and treewidth. We determine which minor-closed graph classes $\mathcal{C}$ have the property that every bag of every refined tree-decomposition of a graph from $\mathcal{C}$ has bounded $\beta$. 

To this end, we introduce the notion of a `container' for a graph parameter. Let $\mathcal{F}=(G_1,G_2,\dots)$ 
be a sequence of graphs with $G_i\preccurlyeq G_j$ and $0<|V(G_i)|<|V(G_j)|$ for all positive integers $i,j$ with $i<j$. For a graph $G$, define \defn{$\phi_{\mathcal{F}}(G)$} to be the minimum positive integer $k$ such that $G_k$ is not a minor of $G$ (which exists because $G_{|V(G)|+1}$ is not a minor of $G$). Note that $\phi_{\mathcal{F}}$ is a minor-monotone graph parameter. We say that $\mathcal{F}$ is a \defn{container} for a graph parameter $\beta$ if $\phi_{\mathcal{F}}$ and $\beta$ are tied. 

Many minor-monotone graph parameters have a container\footnote{Not every minor-monotone graph parameter has a container.
For example, consider the parameter $\beta$ that counts the maximum number of vertices in a connected component of the input graph. The $n$-vertex star $K_{1,n-1}$ and the $n$-vertex path $P_n$ both have a $\beta$ value of $n$.
Thus, if $\mathcal{F}=(G_1,G_2,\dots)$ is a sequence such that $\phi_{\mathcal{F}}$ is tied to $\beta$, then since $\phi_{\mathcal{F}}(G_k)=k+1$ for every positive integer $k$, there is some $k$ such that $\beta(G_k)\geq 4$. 
Additionally, there is some $n$ such that both $P_n$ and $K_{1,n-1}$ contain $G_k$ as a minor.
However every connected minor of $K_{1,n-1}$ is a star and every connected minor of $P_n$ is a path, so no graph is a minor of both $P_n$ and $K_{1,n-1}$, and has a component with more than three vertices. Therefore no such sequence $\mathcal{F}$ exists.}. 
For example, the sequence of  $k\times k$-grid graphs is a container for treewidth, since the $k\times k$-grid has treewidth $k$, and the Grid Minor Theorem of Robertson and Seymour \cite{RS-V} says that every graph with sufficiently large treewidth contains the $k\times k$-grid as a minor (see \citep{CT19} for the best known bound). Similarly, the sequence of all complete binary trees is a container for pathwidth, since the complete binary tree of radius $r$ has pathwidth $\lceil \frac{r}{2}\rceil$ (see \citep{Scheffler89}), and \citet{RS-I} showed that for any tree $T$, every graph with sufficiently large pathwidth contains $T$ as a minor (see \citep{BRST91} for an optimal bound). Treedepth provides a third example: if $P_k$ is the $k$-vertex path, then $(P_1,P_2,\dots)$ is a container for treedepth, since \citet{Sparsity} showed that every graph with sufficiently large treedepth contains $P_k$ as a minor, and $\td(P_k)=\lceil \log_2 (k+1)\rceil$.

Given a sequence $\mathcal{F}=(G_1,G_2,\dots)$ of graphs, define $\mathcal{F}^{+2}$ to be the sequence $(G_1+\overline{K_2},G_2+\overline{K_2},\dots)$, where $G_i+\overline{K_2}$ is the graph obtained from $G_i$ by adding two non-adjacent vertices, each adjacent to every vertex of $G_i$.

\begin{lem}
\label{2BagLemma}
    For any graph $G$, the 2-bag tree-decomposition of $G+\overline{K_2}$ in which each bag consists of $V(G)$ together with one vertex not in $V(G)$ is refined.
\end{lem}

\begin{proof}
Let $B_x$ and $B_y$ be the bags of this tree-decomposition, and let $\{v\}=B_x\setminus V(G)$ and $\{w\}=B_y\setminus V(G)$.
For any pair of vertices $v_1$ and $v_2$ in $B_x$, either $v_1wv_2$ is a path, or $v\in \{v_1,v_2\}$ and $v_1v_2$ is a path.
Thus there is a path from $v_1$ to $v_2$ with no internal vertex in $B_x$. 
By  symmetry, for any pair of vertices $w_1,w_2\in B_y$, there is a path from $w_1$ to $w_2$ with no internal vertex in $B_y$.
The result follows from \cref{thm:refinedcharacterisation}.
\end{proof}

For any integer $k\geq 2$, if $G$ is the $k\times k$ grid, then by \cref{2BagLemma}, $G+\overline{K_2}$ is a $K_7$-minor-free graph that has a refined tree-decomposition with a bag of treewidth at least $k$. 
In contrast, \citet[Theorem~28]{HW26} showed that for any proper minor-closed class $\GG$, in every atomic tree-decomposition of a graph in $\GG$, every bag has bounded treewidth. This example illustrates the distinction between atomic and refined tree-decompositions.

\subsection{Self-Similar and Robust Parameters}

For any graph $H$, we define a \defn{$(H,0)$-core} in a graph $G$ to be a set $S\subseteq V(G)$ such that there is a model of $H$ in $G$ whose branch-sets all intersect $S$. For $i\geq 1$, we recursively define a \defn{$(H,i)$-core} in $G$ to be a set $S\subseteq V(G)$ such that there is a model $(M_v:v\in V(H))$ of $H$ in $G$ such that $M_v\cap S$ is a $(H,i-1)$-core in $G$ for every $v\in V(H)$.
We say that a sequence of graphs $\mathcal{F}=(H_1,H_2,\dots)$ is \defn{self-similar} if $H_1\preccurlyeq H_2\preccurlyeq \dots$ and there is a function $f$ such that for all positive integers $k$ and $d$, the graph $H_{f(k,d)}$ has a $(H_k,d)$-core. 
We say that $\mathcal{F}$ is \defn{robust} if for every positive integer $k$ there is a positive integer $k'$ such that for every $v\in V(H_{k'})$ we have $H_k\preccurlyeq H_{k'}-v$. A parameter is \defn{robust} (or \defn{self-similar}) if it is tied to $\phi_{\mathcal{F}}$ for some robust (or self-similar) sequence $\mathcal{F}$.

The three sequences mentioned above (namely paths, complete binary trees, and grids) are robust and self-similar (see \cref{RobustSelfSimilarSquences}). Thus treedepth, pathwidth and treewidth are robust and self-similar. The next theorem is the main result of this section. It shows that for any robust self-similar sequence $\mathcal{F}$, graphs excluding a certain minor have tree-decompositions in which each bag has $\phi_{\mathcal{F}}$ bounded.

\begin{thm}\label{thm:selfsimilarcase}
    Given a robust and self-similar sequence $\mathcal{F}=(G_1,G_2,\dots)$ and a graph $F\in \mathcal{F}^{+2}$, there is a constant $c$ such that for any bag $B$ of any refined tree-decomposition of any $F$-minor-free graph $G$, we have $\phi_{\mathcal{F}}(G[B])\leq c$.
\end{thm}
\begin{proof}
    Let $k_0$ be the integer such that $F=G_{k_0}+\overline{K_2}$. 
    \citet{Mader68} showed that (for any graph $F$) there is an integer $d_F$ such that every graph with average degree at least $d_F$ contains $F$ as a minor; see \cite{Kostochka82,Kostochka84,Thomason84,Thomason01,ReedWood16,TW22,MT05,NRTW20} for explicit bounds. Let $k_1$ be an integer such that $G_{k_0}\preccurlyeq G_{k_1}-v$ for all $v\in V(G_{k_1})$. 
    Let $d$ be a sufficiently large integer, to be computed later, and let $n$ be an integer such that $G_n$ contains a $(G_{k_1},d)$-core $S$ witnessed by a model $\mathcal{M}$ of $G_{k_1}$.
    For each $t \in \{1,\dots,d\}$ and each sequence $\mathbf{v}=(v_1,v_2,\dots, v_t)\in V(G_{k_1})^t$, we recursively construct a $(G_{k_1},d-t)$-core $S_{\mathbf{v}}$ and a model $\mathcal{M}_{\mathbf{v}}$ of $G_{k_1}$ as follows.
    If $t=1$, then for each $v_1\in V(G_{k_1})$ let $S_{(v_1)}$ be the intersection of $S$ and the branch-set of $\mathcal{M}$ corresponding to $v_1$, and let $\mathcal{M}_{(v_1)}$ be a model of $G_{k_1}$ witnessing that $S_{(v_1)}$ is a $(G_{k_1},d-1)$-core.
    If $t>1$, then for each $\mathbf{v}=(v_1,\dots, v_t)\in V(G_{k_1})^t$ let $S_{\mathbf{v}}$ be the intersection of $S_{(v_1,\dots,v_{t-1})}$ and the branch-set of $\mathcal{M}_{(v_1,\dots,v_{t-1})}$ corresponding to $v_t$, and let $\mathcal{M}_{\mathbf{v}}$ be a model of $G_{k_1}$ witnessing that $S_{\mathbf{v}}$ is a $(G_{k_1},d-t)$-core.
    By construction, $\mathcal{S}:=\{S_{\mathbf{v}}:\mathbf{v}\in V(G_{k_1})^d\}$ is a set of $|V(G_{k_1})|^d$ disjoint $(G_{k_1},0)$-cores in $G_n$.
    Choose a representative vertex $x_{\mathbf{v}}$ in $S_{\mathbf{v}}$ for each $\mathbf{v}\in V(G_{k_1})^d$.

    Let $G$ be an $F$-minor-free graph, and suppose for contradiction that there is a refined tree-decomposition of $G$ with a bag $B$ such that $G_n\preccurlyeq G[B]$.
    Let $G'=G\llbracket B\rrbracket$, and note that $G'$ is also $F$-minor-free since $G'\preccurlyeq G$.
    
    \begin{claim}
        There is no external vertex $x$ of $G'$ such that for some model $\mathcal{M}^*$ of $G_{k_1}$ in $G[B]$, the set $N_{G'}(x)$ contains a vertex of every branch-set of $\mathcal{M}^*$.
    \end{claim}
    \begin{proofclaim}
    By property \ref{torso2} in \cref{lem:torso}, there is some $v\in B\setminus N_{G'}(x)$.
    Let $G''$ be obtained from $G'$ by contracting every edge between $v$ and an external vertex of $G'$, and note that $B\setminus \{v\}\subseteq N_{G''}(v)$ by property \ref{torso1} in \cref{lem:torso}.
    By the choice of $k_1$, we have $G_{k_0}\preccurlyeq G_{k_1}-w$ for any $w\in V(G_{k_1})$.
    Hence there is a model $\mathcal{M}'$ of $G_{k_0}$ in $G[B]$ whose branch-sets are unions of branch-sets in $\mathcal{M}^*$ that do not contain $v$.
    We can now extend $\mathcal{M}'$ to a model of $F$ in $G''$ by adding the singleton branch-sets $\{x\}$ and $\{v\}$, contradicting that $G'$ is $F$-minor-free.
    \end{proofclaim}
    Let $\mathcal{M}^*:=(M^*_{v}:v\in V(G_n))$ be a model of $G_n$ in $G[B]$, and for each $\mathbf{v}\in V(G_{k_1})^d$ let $u_{\mathbf{v}}$ be a vertex in the branch-set $M^*_{x_{\mathbf{v}}}$.

    Now for an external vertex $x$ of $G'$, consider the set $V_x:=\{\mathbf{v}\in V(G_{k_1})^d: xu_{\mathbf{v}}\in E(G')\}$. 
    Suppose for contradiction that for some $t\in \{0,1,\dots,d-1\}$, some $(v_1,\dots,v_t)\in V(G_{k_1})^t$ and every $w\in V(G_{k_1})$ there exists $\mathbf{v}\in V_x$ with $(v_1,\dots, v_t,w)$ as a prefix.
    By construction $x_{\mathbf{v}}$ is in $S_{\mathbf{v}'}$ for every non-empty prefix $\mathbf{v}'$ of $\mathbf{v}$, and so for every branch-set $M_{(v_1,\dots,v_t),w}$ of $\mathcal{M}_{(v_1,\dots,v_t)}$ (where $\mathcal{M}_{\emptyset}:=\mathcal{M}$), there is a vertex $v\in M_{(v_1,\dots,v_t),w}$ such that $N_{G'}(x)$ contains a vertex of $M^*_v$. 
For each $w\in V(G_{k_1})$, let $M'_w:=\bigcup \{M^*_v:v\in M_{(v_1,\dots,v_t),w}\}$. 
So $\mathcal{M}':=(M'_w:w\in V(G_{k_1}))$ is a model of $G_{k_1}$ in $G[B]$. 
Moreover, $N_{G'}(x)$ contains a vertex of $M'_w$ for each $w\in V(G_{k_1})$, which contradicts the above claim. Hence for each $t\in \{0,1,\dots,d-1\}$ and each $(v_1,\dots,v_t)\in V(G_{k_1})^t$, there is at least one $w\in V(G_{k_1})$ such that $(v_1,\dots,v_t,w)$ is not a prefix of any $\mathbf{v}\in V_x$. 

    Now consider a vector $\mathbf{v}\in V(G_{k_1})^d$ chosen uniformly at random, and for each $i\in \{0,1,\dots,d\}$ let $\mathbf{v}_i$ be the prefix of $\mathbf{v}\in V(G_{k_1})^i$, and let $A_{x,i}$ be the event that $\mathbf{v}_i$ is a prefix of some $\mathbf{v}'\in V_x$.
    Since $V_x\subseteq V(G_{k_1})^d$, the event that $\mathbf{v}\in V_x$ corresponds to $A_{x,d}$. Hence $P(A_{x,d}) = |V_x| / |V(G_{k_1})|^d$. Unless it is equal to $0$, $P(A_{x,d})$ can also be written in terms of conditional probabilities as follows:
    \[P(A_{x,d})=P(A_{x,0})\cdot P(A_{x,1}\vert A_{x,0})\cdot P(A_{x,2}\vert A_{x,1})\cdot\dots \cdot P(A_{x,d}\vert A_{x,d-1}).\]
    Note that $P(A_{x,0})=1$ unless $V_x$ is empty.
    When $P(A_{x,d})\neq 0$ we have $V_x\neq \emptyset$, and by the previous paragraph, $P(A_{x,i}\vert A_{x,i-1})\leq (|V(G_{k_1})|-1)/|V(G_{k_1})|$ for each $i\in \{1,\dots, d\}$.
    Hence $P(A_{x,d})\leq (|V(G_{k_1})|-1)^d/|V(G_{k_1})|^d$, and so $|V_x|\leq (|V(G_{k_1})|-1)^d$.

    Let $X:=\{x_1,\dots,x_{\ell}\}$ be the set of external vertices of $G'$ and let $U:=\{u_{\mathbf{v}}: \mathbf{v}\in V(G_{k_1})^d\}$. 
    Let $H_0:=G'[X\cup U]$ and for each $i\in \{1,\dots, \ell\}$ let $H_i$ be a minor of $H_{i-1}$ with exactly one fewer vertex obtained as follows. If $x_i$ is isolated, then delete $x_i$. Otherwise, choose an edge incident to $x_i$ to contract in order to maximise the number of edges in the resulting graph.   
    Let $E_i$ be the set of edges of $H_i$ that are non-edges of $H_{i-1}$, and let $N_i$ be the set of non-edges in $H_i[N_{H_{i-1}}(x_i)]$.
    By the choice of $H_i$, the edges in $E_i$ form a star whose centre is a vertex of maximum degree in the graph induced by $E_i\cup N_i$.
    Thus $|N_i|\leq \frac{1}{2}|E_i|\cdot |N_{H_{i-1}}(x_i)|=\frac{1}{2}|E_i|\cdot |V_{x_i}|\leq \frac{1}{2}|E_i|\cdot (|V(G_{k_1})|-1)^d$.
    Now $H_\ell$ is a minor of $G'$ rooted on $U$ with $\sum_{i=1}^\ell |E_i|$ edges whose complement $\overline{H_\ell}$ has at most $\sum_{i=1}^\ell |N_i|$ edges. 
    Thus $|E(\overline{H_\ell})|\leq \frac{1}{2}(|V(G_{k_1})|-1)^d |E(H_\ell)|$.
    This means that $H_\ell$ is a graph on $|U|=|V(G_{k_1})|^d$ vertices with average degree at least $(|V(G_{k_1})|^d-1)/(1+(|V(G_{k_1})|-1)^d)$. 
    For sufficiently large $d$, $H_\ell$ has average degree greater than $d_F$, implying $F\preccurlyeq H_\ell \preccurlyeq G'\preccurlyeq G$, which is the desired contradiction.
    \end{proof}

\cref{thm:selfsimilarcase} and \cref{2BagLemma} together imply the following general result, which is the main result of this section.

\begin{cor}\label{mainextcor}
 Given a robust and self-similar sequence $\mathcal{F}=(G_1,G_2,\dots)$ and a minor-closed graph class $\mathcal{C}$, the following statements are equivalent.
 \begin{enumerate}[(A)]
    \item\label{1stform} There is a constant $c$ such that for any bag $B$ of any refined tree-decomposition of any $G\in\mathcal{C}$, we have $\phi_{\mathcal{F}}(G[B])\leq c$.
    \item\label{2ndform} $\mathcal{C}$ does not contain every graph in $\mathcal{F}^{+2}$.
 \end{enumerate}
\end{cor}
\begin{proof}
    It follows immediately from \cref{thm:selfsimilarcase} that \ref{2ndform} implies \ref{1stform}.
    Suppose that $\mathcal{C}$ contains every graph in $\mathcal{F}^{+2}$.
    Pick any positive integer $c$, and consider the refined tree-decomposition $\mathcal{D}$ of $G_{c}+\overline{K_2}$ guaranteed by \cref{2BagLemma}.
    For each bag $B$ of $\mathcal{D}$, $G_c\subseteq (G_{c}+\overline{K_2})[B]$, so by definition $\phi_{\mathcal{F}}((G_{c}+\overline{K_2})[B])> c$.
    Thus \ref{1stform} implies \ref{2ndform}.
\end{proof}

The following simple observation will be useful in analysing $(H,d)$-cores.

\begin{obs}\label{obs:ss}
    Let $H,G',G$ be graphs such that $H \preccurlyeq G' \preccurlyeq G$. If $S$ is an $(H,d)$-core in $G'$, $M:=(M_v:v\in V(G'))$ is a model of $G'$ in $G$ and $S^*\subseteq V(G)$ is a superset of $\bigcup \{M_v:v\in S\}$, then $S^*$ is an $(H,d)$-core in~$G$.
\end{obs}

\begin{proof}
    Assume the contrary and consider a counterexample minimising $d$.
    Let $\mathcal{M}':=(M'_v:v\in V(H))$ be a model of $H$ in $G'$ that witnesses that $S$ is an $(H,d)$-core, and for each $v\in V(H)$ let $M^*_v:=\bigcup \{M_w:w\in M'_v\}$.
    
    Suppose that $d=0$, and consider an arbitrary $v\in V(H)$. Since $\mathcal{M}'$ witnesses that $S$ is an $(H,0)$-core in $G'$, there is a vertex $s$ in $M'_v\cap S$, and so $M_s\subseteq M^*_v\cap S^*$. 
    Thus $\mathcal{M}^*$ witnesses that $S^*$ is a $(H,0)$-core.

    Now suppose that $d\geq 1$.
    Consider an arbitrary $v\in V(H)$. Since $\mathcal{M}'$ witnesses that $S$ is an $(H,d)$-core in $G'$, we have that $S\cap M'_v$ is an $(H,d-1)$-core in $G'$.
    Since our counterexample minimises $d$, any subset of $V(G)$ containing $\bigcup \{M_v:v\in S\cap M'_v\}$ is an $(H,d-1)$-core in $G$.
    By definition, $S^*\cap M^*_v\supseteq \bigcup \{M_v:v\in S\cap M'_v\}$, and so $S^*$ is an $(H,d)$-core in $G$.
\end{proof}

    \begin{lem}\label{lem:nestedmodels}
    Let $\mathcal{F}:=(G_1,G_2,\dots)$ be a sequence of graphs such that $G_i\preccurlyeq G_{i+1}$ for all $i$, and for every pair of positive integers $i$ and $j$ there is an integer $g(i,j)$ such that $G_{g(i,j)}$ contains a model of $G_i$ whose branch-sets each induce subgraphs that contain $G_j$ as a minor.
    Then $\mathcal{F}$ is self-similar.
    \end{lem}
    \begin{proof}
    Given any positive integer $k$, we show by induction on $d$ that there is an integer $f(k,d)$ such that $G_{f(k,d)}$ has a $(G_k,d)$-core.
    A graph has a $(G_k,0)$-core if and only if it contains $G_k$ as a minor, so we may set $f(k,0):=k$.
    Now suppose $d\geq 1$ and $f(k,d-1)$ exists, and set $f(k,d):=g(k,f(k,d-1))$.
    By definition, $G_{f(k,d)}$ contains a model $(M_v:v\in V(G_k))$ of $G_k$ whose branch-sets each induce graphs that contain $G_{f(k,d-1)}$ as a minor.
    Now for each $v\in V(G_k)$ there is a $(G_k,d-1)$-core $S_v$ in $G[M_v]$, which is also a $(G_k,d-1)$-core in $G$.
    It follows that $\bigcup \{S_v:v\in V(G_k)\}$ is a $(G_k,d)$-core, as required.
    \end{proof}

    \begin{thm}\label{RobustSelfSimilarSquences}
    Each of the following is a robust, self-similar sequence:
    \begin{enumerate}
        \item $\mathcal{F}_1:=(P_1,P_2,\dots)$, where $P_i$ is the $i$-vertex path,
        \item $\mathcal{F}_2:=(T_1, T_2,\dots)$, where $T_i$ is the complete binary tree of radius $i$,
        \item $\mathcal{F}_3:=(P_1\square P_1,P_2\square P_2,\dots )$,
        \item $\mathcal{F}_4:=(C_3\square C_3, C_4\square C_4, \dots)$
    \end{enumerate}
    \end{thm}
    \begin{proof}
        Let $k$ be a positive integer.
        Note that $P_{2k}$ contains two vertex disjoint copies of $P_k$. Thus for any $v\in V(P_{2k})$, $P_k$ is a minor of $P_{2k}-v$. Hence $\mathcal{F}_1$ is robust.
        Likewise $T_{k+1}$ contains two vertex disjoint copies of $T_k$ and $P_{2k}\square P_{2k}$ contains two vertex disjoint copies of $P_k\square P_k$. Hence $\mathcal{F}_1$, $\mathcal{F}_2$ and $\mathcal{F}_3$ are robust.

        Given any pair of positive integers $i$ and $j$, it is easy to see that there is a model of $P_i$ in $P_{ij}$ whose branch-sets each induce subgraphs isomorphic to $P_j$.
        Likewise, there is a model of $T_i$ in $T_{ij}$ whose branch-sets each induce subgraphs isomorphic to $T_j$ and there is a model of $P_i\square P_i$ in $P_{ij} \square P_{ij}$ whose branch-sets each induce subgraphs isomorphic to $P_j\square P_j$.
        Hence $\mathcal{F}_1, \mathcal{F}_2$ and $\mathcal{F}_3$ are self-similar by \cref{lem:nestedmodels}.

        Let $V(C_{k+3}):=\{w_1,\dots, w_{k+3}\}$, and consider an arbitrary vertex $(w_i,w_j)\in V(C_{k+3}\square C_{k+3})$.
        Note that for each $\ell\in \{1,\dots ,k+3\}$         the graphs 
        \begin{align*}
            &O_{\ell}:=(C_{k+3}\square C_{k+3})[\{(w_{\ell},w_1),\dots ,(w_{\ell},w_{k+3})\}] \text{ and } \\
            &O^{\ell}:=(C_{k+3}\square C_{k+3})[\{(w_1,w_{\ell}),\dots, (w_{k+3},w_{\ell})\}]
        \end{align*}
        
        are both cycles, and that 
        \[\bigcup\{O_{\ell}:\ell \in \{1,\dots,k+3\}/ \{i\}\}\cup \bigcup\{O^{\ell}:\ell \in \{1,\dots,k+3\}/\{j\}\}\] 
        is isomorphic to a subdivision of $C_{k+2}\square C_{k+2}$.
        Hence $C_{k+2}\square C_{k+2}$ is a minor of $C_{k+3}\square C_{k+3}-(w_i,w_j)$, and so $\mathcal{F}_4$ is robust.

        Given positive integers $d$ and $t$ and $\ell$ with $(d+2)^t\leq \ell+2$, let $\mathcal{S}_{t,d,\ell}$ be the set of subsets of $V(C_{\ell+2}\square C_{\ell+2})$ of the form $\{(w_i,w_j):i\in I,j\in I'\}$, where $I$ and $I'$ are both intervals of $(d+2)^t$ consecutive integers in $\{1,\dots,\ell+2\}$.
        To prove that $\mathcal{F}_4$ is self-similar, we show by induction on $t$ that every set in $\mathcal{S}_{t,d,\ell}$ is a 
        $(C_{d+2}\square C_{d+2},t-1)$-core in $C_{\ell+2}\square C_{\ell+2}$.
        Let $I$ and $I'$ be intervals of $(d+2)^t$ consecutive integers in $\{1,\dots,\ell+2\}$, and let $S:=\{(w_i,w_j):i\in I,j\in I'\}$.

        When $t=1$, the graph $H:=\bigcup \{O_{i}:i\in I\}\cup \bigcup \{O^i:i\in I'\}$ is a subdivision of $C_{d+2}\square C_{d+2}$, and there is a model of $C_{d+2}\square C_{d+2}$ in $H$ such that every branch-set contains exactly one vertex of $S$. Thus every set in $\mathcal{S}_{1,d,\ell}$ is a $(C_{d+2}\square C_{d+2},0)$-core.

        When $t\geq 2$, $S$ induces a subgraph isomorphic to $P_{(d+2)^t}\square P_{(d+2)^t}$ that contains a model $\mathcal{M}$ of $P_{d+2}\square P_{d+2}$ whose branch-sets are all in $\mathcal{S}_{t-1,d,\ell}$.
        It is straightforward to find a model $\mathcal{M}'$ of $C_{d+2}\square C_{d+2}$ in $C_{\ell+2}\square C_{\ell+2}$ whose branch-sets are all supersets of branch-sets of $\mathcal{M}$.
        By our inductive hypothesis and \cref{obs:ss}, $S$ is a $(C_{d+2}\square C_{d+2},t-1)$-core, and so every set in $\mathcal{S}_{t,d,\ell}$ is a $(C_{d+2}\square C_{d+2},t-1)$-core, as required.
    \end{proof}
    
We now give several examples of the above machinery. 

\begin{thm}
    \label{DoubleApexPath}
    Given a minor-closed class $\mathcal{C}$, the following statements are equivalent.
    \begin{enumerate}
    \item\label{1stformtd} There is a constant $c$ such that for any bag $B$ of any refined tree-decomposition of any $G\in\mathcal{C}$, we have $\td(G[B])\leq c$.
    \item\label{2ndformtd} There is a path $P$ such that $P+\overline{K_2}$ is not in $\mathcal{C}$.
    \end{enumerate}
\end{thm}

\begin{proof}
    By part 1 of \cref{RobustSelfSimilarSquences}, the sequence $\mathcal{F}:=(P_1,P_2,\dots )$, $P_i$ is the $i$-vertex path, is robust and self-similar. Thus, by \cref{mainextcor} it suffices to show that $\mathcal{F}$ is a container for treedepth.
    By a result of \citet[Proposition 6.1]{Sparsity}, for every positive integer $n$ the treedepth of a graph with no $P_n$ minor is at most $n-1$.
    Thus for every graph $G$ with $\phi_{\mathcal{F}}(G)=n$, we have $\td(G)\leq n-1$.
    \citet{Sparsity} also showed that every graph containing $P_n$ as a minor has treedepth at least $\lceil \log_2(n+1)\rceil$.
    Thus treedepth is tied to $\phi_{\mathcal{F}}$, as required.
\end{proof}

Note that for a path $P$, there is a straight line drawing of $P+\overline{K_2}$ in $\mathbb{R}^2$ with the vertices of $P$ on the positive $x$-axis and the other two vertices embedded at $(0,1)$ and $(0,-1)$.
In particular, $P+\overline{K_2}$ is planar, and so there is an integer $c$ such that every $(P+\overline{K_2})$-minor-free graph has treewidth at most $c$ (by (1.5) in \cite{RS-V}).
This trivialises the above result for optimal tree-decompositions, since the number of vertices of a graph is an upper bound for its treedepth. The strength of \cref{DoubleApexPath} is that it applies to all refined tree-decompositions of a $(P+\overline{K_2})$-minor-free graph $G$, which can have width much larger than $\tw(G)$.

\begin{thm}
\label{DoubleApexForest}
    Given a minor-closed class $\mathcal{C}$, the following statements are equivalent.
    \begin{enumerate}
    \item\label{1stformpw} There is a constant $c$ such that for any bag $B$ of any refined tree-decomposition of any $G\in\mathcal{C}$, we have $\pw(G[B])\leq c$.
    \item\label{2ndformpw} There is a complete binary tree $T$ such that $T+\overline{K_2}$ is not in $\mathcal{C}$.
    \end{enumerate}
\end{thm}
\begin{proof}
Let $T_i$ be the complete binary tree of radius $i$. By part 2 of \cref{RobustSelfSimilarSquences}, the sequence $\mathcal{F}:=(T_1,T_2,\dots )$ is robust and self-similar. Thus, by \cref{mainextcor} it suffices to show that $\mathcal{F}$ is a container for pathwidth. For every positive integer $n$, every graph of sufficiently large pathwidth contains $T_n$ as a minor, by the result of \citet{RS-I} discussed earlier in this section. Conversely, every graph that contains $T_n$ as a minor has pathwidth at least $\lceil\frac{n}{2}\rceil$, as was shown by \citet{Scheffler89}. Thus pathwidth is tied to $\phi_{\mathcal{F}}$, as required. 
\end{proof}

It follows from Euler's formula that every graph of Euler genus $g$ is $K_{3,2g+3}$-minor-free, and $K_{3,2g+3} \subseteq K_{1,2g+3}+\overline{K_2}$, where $K_{1,2g+3}$ is the star graph with $2g+3$ leaves. Since $K_{1,2g+3}$ is a minor of a complete binary tree with at most $4g+5$ leaves, \cref{DoubleApexForest}  implies:

\begin{cor}
For any integer $g\geq 0$ there exists an integer $c$ such that for any graph $G$ of Euler genus $g$, for any bag $B$ of any refined tree-decomposition of $G$, we have $\pw(G[B])\leq c$. In particular, every optimal tree-decomposition of $G$ has a refinement with this property.
\end{cor}

\begin{thm}
    Given a minor-closed class $\mathcal{C}$, the following statements are equivalent.
    \begin{enumerate}
    \item\label{1stformtw} There is a constant $c$ such that for any bag $B$ of any refined tree-decomposition of any $G\in\mathcal{C}$, we have $\tw(G[B])\leq c$.
    \item\label{2ndformtw} There is a planar graph $G$ such that $G+\overline{K_2}$ is not in $\mathcal{C}$.
    \end{enumerate}
\end{thm}
\begin{proof}
    By part 3 of \cref{RobustSelfSimilarSquences}, the sequence $\mathcal{F}:=(P_1\square P_1,P_2\square P_2,\dots )$, where $P_i$ is the $i$-vertex path, is robust and self-similar. Thus, by \cref{mainextcor} it suffices to show that $\mathcal{F}$ is a container for treewidth.
    This follows from the fact that the $P_n\square P_n$ has treewidth $n$, together with the Grid Minor Theorem of \citet{RS-V}.
\end{proof}
\begin{thm}
\label{DoubleApexPlanar}
For any planar graph $H$ there exists an integer $c$ such that for any $(H+\overline{K_2})$-minor-free graph $G$, for any bag $B$ of any refined tree-decomposition of $G$, we have $\tw(G[B])\leq c$. In particular, every refined optimal tree-decomposition of $G$ has this property. 
Moreover, for every positive integer $c$ there is a planar graph $H$ and a bag $B$ of a refined tree-decomposition of $H+\overline{K_2}$ such that $\tw((H+\overline{K_2})[B])>c$.
\end{thm}

\begin{proof}
Let $H_i$ be $P_i\square P_i$. 
By \cref{RobustSelfSimilarSquences}, $\mathcal{F}:=(H_1, H_2,\dots)$ is a robust, self-similar sequence.
\citet{RST94} proved that $H$ is a minor of $H_i$ for some $i\leq 2|V(H)|$.
Let $F:=H_i+\overline{K_2}$. So $F\in\mathcal{F}^{+2}$. 
Any $H+\overline{K_2}$-minor-free graph is $F$-minor-free. By \cref{thm:selfsimilarcase}, there exists $c=c(H)$ such that for 
any $(H+\overline{K_2})$-minor-free graph $G$, for any bag $B$ of any refined tree-decomposition of $G$, we have 
$\phi_{\mathcal{F}}(G[B])\leq c$. That is $H_c$ is not a minor of $G[B]$. 
By the Grid Minor Theorem of \citet{RS-V}, $\tw(G[B])$ is at most some function of $c$.

To prove the final claim of the theorem, let $H$ be the $(c+1)\times(c+1)$ grid, and let $\DD$ be the refined tree-decomposition of $H+\overline{K_2}$ given by \cref{2BagLemma}.
For each bag $B$ of $\DD$, the $(c+1)\times(c+1)$-grid is a subgraph of $(H+\overline{K_2})[B]$. Since the $(c+1)\times(c+1)$-grid has treewidth $c+1$, we have $\tw((H+\overline{K_2})[B])\geq c+1>c$.
\end{proof}

The sequence $\mathcal{F}_4:=(C_3\square C_3,C_4\square C_4,\dots)$ is not a container for any commonly studied parameter as far as we know, but we include it to help demonstrate the full power of \cref{thm:selfsimilarcase}.

\subsection{Excluded Ladders}

Given a robust, self-similar parameter $\beta$, \cref{mainextcor} characterises the minor-closed graph classes $\mathcal{C}$ for which $\beta$ is bounded on bags of refined tree-decompositions of graphs from $\mathcal{C}$. We now give an example to show that not all parameters for which this characterisation holds are robust and self-similar.

For a positive integer $n$, the $n \times 2$ grid $P_n\square K_2$ is called the \defn{ladder} of order $n$, denoted \defn{$L_n$}. The sequence $(L_1,L_2,\dots)$ is not self-similar. However, we still prove a variant of \Cref{thm:selfsimilarcase}.
The \defn{$2$-treedepth} of a graph $G$, denoted $\mathdefn{\td_2(G)}$, is defined recursively as follows: 
\begin{enumerate}
    \item $\td_2(K_1):=1$,
    \item if $G$ is not a block, then $\td_2(G)$ is the maximum $2$-treedepth of a block of $G$, and
    \item if $G$ is a block and $|V(G)|\geq 2$, then $\td_2(G):=\min\{1+\td_2(G-v):v\in V(G)\}$.
\end{enumerate} 
\citet{HJMSW22} proved that $(L_1,L_2,\dots)$ is a container for $2$-treedepth.

\begin{thm}
    For any positive integer $d$, there exists an integer $c$ such that for any $(L_d+\overline{K_2})$-minor-free graph $G$, for any bag $B$ of any refined tree-decomposition of $G$, $\td_2(G[B]) \leq c$. 
    Moreover for any integer $c$ there is a positive integer $d$ such that there is a refined tree-decomposition of $L_d+\overline{K_2}$ with a bag $B$ satisfying $\td_2((L_d+\overline{K_2})[B])>c$.
\end{thm}
\begin{proof}
    For a given positive integer $d$, we fix a large enough integer $c$, to be determined later.
    Let $F := L_d+\overline{K_2}$.
    Let $G$ be an $F$-minor-free graph, and suppose for contradiction that there is a refined tree-decomposition of $G$ with a bag $B$ such that $L_c\preccurlyeq G[B]$.
    Let $\mathcal{M}$ be a model of $L_c$ in $G[B]$.
    Let $G'=G\llbracket B\rrbracket$, and note that $G'$ is also $F$-minor-free since $G'\preccurlyeq G$.
    
    \begin{claim}
        There is no external vertex $x$ of $G'$ such that for any model $\mathcal{M}$ of $L_{2d}$ in $G[B]$, the set $N_{G'}(x)$ contains a vertex of every branch-set of $\mathcal{M}$.
    \end{claim}
    \begin{proofclaim}
    By property \ref{torso2} in \cref{lem:torso}, there is some $v\in B\setminus N_{G'}(x)$.
    Let $G''$ be obtained from $G'$ by contracting every edge between $v$ and an external vertex of $G'$, and note that $B\setminus \{v\}\subseteq N_{G''}(v)$ by property \ref{torso1} in \cref{lem:torso}.
    By the basic properties of ladders, $L_{d} \preccurlyeq L_{2d}-w$ for any $w \in V(L_{2d})$.
    Hence there is a model $\mathcal{M}'$ of $L_{d}$ in $G[B]$ whose branch-sets do not contain $v$.
    We can now extend $\mathcal{M}'$ to a model of $F$ in $G''$ by adding the singleton branch-sets $\{x\}$ and $\{v\}$.
    This contradicts $G'$ being $F$-minor-free.
    \end{proofclaim}

    Let $\mathcal{M} = (M_v : v \in V(L_c))$ be a model of $L_c$ in $G'$.
    Let $a_1,\dots,a_c$ be the vertices of one of the rows of $L_c$ and let $b_1,\dots,b_c$ be the vertices of the other row of $L_c$ in the natural order so that $a_ib_i \in E(L_c)$ for every $i \in \{1,\dots,c\}$.
    For every $i \in \{1,\dots,c\}$, we pick vertices $u_i \in M_{a_i}$ and $v_i\in M_{b_i}$ arbitrarily.

    By property \ref{torso1} in \cref{lem:torso}, for any $u,v \in B$, either $uv \in E(G')$ or there exists an external vertex $x$ of $G'$ with $xu,xv \in E(G')$.

    Define the following colouring of $\binom{\{1,\dots,c\}}{2}$. For a pair of indices $(i,j)$ with $i < j$ as above, let \defn{$\alpha(i,j)$} be $\star$ when $u_iv_j \in E(G')$, and otherwise, let \defn{$\alpha(i,j)$} be an arbitrary external vertex $x$ of $G'$ with $xu,xv \in E(G')$.

    By the canonical Ramsey theorem of \citet{ErdosRado50}, for $c$ large enough, there exists a set of indices $i_1,\dots,i_{4d+4}$, without loss of generality\footnote{We justify this renumbering step in the next paragraph.} renumbered to $1,\dots,4d+4$, such that at least one of the following holds:
    \begin{enumerate}
        \item there exists a colour $\beta$ such that for all integers $i$ and $j$ with $1 \leq i < j \leq 4d+4$, $\alpha(i,j) = \beta$ (\defn{monochromatic} set);
        \item for all integers $i$ and $j$ with $1 \leq i < j \leq 4d+4$, $\alpha(i,j) = \beta_{i,j}$ and these values are pairwise distinct (\defn{rainbow} set);
        \item there exist pairwise distinct colours $\beta_1,\dots,\beta_{4d+3}$ such that for all integers $i$ and $j$ with $1 \leq i < j \leq 4d+4$, $\alpha(i,j) = \beta_i$ (\defn{lexicographic} set);
        \item there exist pairwise distinct colours $\beta_2,\dots,\beta_{4d+4}$ such that for all integers $i$ and $j$ with $1 \leq i < j \leq 4d+4$, $\alpha(i,j) = \beta_j$ (\defn{lexicographic} set).
    \end{enumerate}
    We now explain why the renumbering was indeed without loss of generality. If $i_j < i_{j+1}$ for some index $j$, then merge $M_{a_{i_j}}$ with all the branch sets $M_{a_s}$ with $i_j < s < i_{j+1}$, and similarly merge $M_{b_{i_j}}$ with all the branch sets $M_{b_s}$ with $i_j < s < i_{j+1}$.
    Now, the situation is the same as if the found indices were consecutive.   
    
    Consider the cases one-by-one.

    \textit{Case 1 (monochromatic set).}
    If $\beta \neq \star$, then since $4d+4 \geq 2d$, we obtain a contradiction with the claim.
    Thus, we may assume that $\beta = \star$.
    By merging $M_{a_i}$ with $M_{b_i}$ for every $i \in \{1,\dots,4d+4\}$, we obtain a model of $K_{4d+4}$.
    Indeed, $M_{a_i}$ is adjacent to $M_{b_j}$ for every pair $i$ and $j$ with $1 \leq i < j \leq 4d+4$.
    Since $|V(F)| = 2d+2 \leq 4d+4$, this is a contradiction.

    \textit{Case 2 (rainbow set).} 
    At most one of the colours among $\beta_{i,j}$ is $\star$.
    Without loss of generality, suppose that it is $\beta_{4d+3,4d+4}$.
    Again, we merge $M_{a_i}$ with $M_{b_i}$ for every $i \in \{1,\dots,4d+4\}$.
    Next, for every $i \in \{1,\dots,2d+2\}$, we add the set $\{\beta_{i,j} : j \in \{2d+3, \dots , 4d+4\}\}$ to $M_{a_i} \cup M_{b_i}$.
    This gives a model of $K_{2d+2,2d+2}$ in $G'$, however, $K_{2d+2}$ is a minor of $K_{2d+2,2d+2}$, hence, this again contradicts $G'$ being $F$-minor-free.
    
    \textit{Case 3 (lexicographic set).}
    At most one of the colours among $\beta_{i}$ is $\star$.
    Without loss of generality, suppose that it is $\beta_{4d+3}$.
    We merge $M_{a_i}$ with $M_{b_i}$ for every $i \in \{1,\dots,4d+4\}$.
    Next, for every $i \in \{1,\dots,2d+2\}$, we add $\{\beta_{i}\}$ to $M_{a_i} \cup M_{b_i}$.
    This again gives a model of $K_{2d+2,2d+2}$ in $G'$, however, $K_{2d+2}$ is a minor of $K_{2d+2,2d+2}$, hence, this again contradicts $G'$ being $F$-minor-free.
    The argument for case 4 is completely symmetric.
    
    Since each case gives a contradiction, we obtain the final contradiction, which completes the proof.
\end{proof}

{\fontsize{10pt}{11pt}\selectfont
\def\soft#1{\leavevmode\setbox0=\hbox{h}\dimen7=\ht0\advance \dimen7 by-1ex\relax\if t#1\relax\rlap{\raise.6\dimen7 \hbox{\kern.3ex\char'47}}#1\relax\else\if T#1\relax \rlap{\raise.5\dimen7\hbox{\kern1.3ex\char'47}}#1\relax \else\if d#1\relax\rlap{\raise.5\dimen7\hbox{\kern.9ex \char'47}}#1\relax\else\if D#1\relax\rlap{\raise.5\dimen7 \hbox{\kern1.4ex\char'47}}#1\relax\else\if l#1\relax \rlap{\raise.5\dimen7\hbox{\kern.4ex\char'47}}#1\relax \else\if L#1\relax\rlap{\raise.5\dimen7\hbox{\kern.7ex \char'47}}#1\relax\else\message{accent \string\soft \space #1 not defined!}#1\relax\fi\fi\fi\fi\fi\fi}

}

\begin{thebibliography}{52}
\providecommand{\natexlab}[1]{#1}
\providecommand{\msn}[1]{MR:\,\href{http://www.ams.org/mathscinet-getitem?mr=MR{#1}}{#1}}
\providecommand{\ZBL}[1]{Zbl:\,\href{https://www.zentralblatt-math.org/zmath/en/search/?q=an:#1}{#1}}
\providecommand{\url}[1]{\texttt{#1}}
\providecommand{\urlprefix}{}
\expandafter\ifx\csname urlstyle\endcsname\relax
  \providecommand{\doi}[1]{doi:\discretionary{}{}{}#1}\else
  \providecommand{\doi}{doi:\discretionary{}{}{}\begingroup \urlstyle{rm}\Url}\fi

\bibitem[{Abrishami et~al.(2025)Abrishami, Alecu, Chudnovsky, Hajebi, and Spirkl}]{AACHS25}
\textsc{Tara Abrishami, Bogdan Alecu, Maria Chudnovsky, Sepehr Hajebi, and Sophie Spirkl}.
\newblock \href{https://arxiv.org/abs/2307.13684}{Induced subgraphs and tree decompositions {X}. {T}owards logarithmic treewidth for even-hole-free graphs}.
\newblock 2025, arXiv:2307.13684.

\bibitem[{Barrera-Cruz et~al.(2019)Barrera-Cruz, Felsner, M\'{e}sz\'{a}ros, Micek, Smith, Taylor, and Trotter}]{BFMMSTT19}
\textsc{Fidel Barrera-Cruz, Stefan Felsner, Tam\'{a}s M\'{e}sz\'{a}ros, Piotr Micek, Heather Smith, Libby Taylor, and William~T. Trotter}.
\newblock \href{https://doi.org/10.1016/j.jctb.2019.02.003}{Separating tree-chromatic number from path-chromatic number}.
\newblock \emph{J. Combin. Theory Ser. B}, 138:206--218, 2019.

\bibitem[{Berger and Seymour(2024)}]{BS24}
\textsc{Eli Berger and Paul Seymour}.
\newblock \href{https://doi.org/10.1007/s00493-024-00088-1}{Bounded-diameter tree-decompositions}.
\newblock \emph{Combinatorica}, 44(3):659--674, 2024.

\bibitem[{Biedl and Vel{\'{a}}zquez(2013)}]{BV13}
\textsc{Therese Biedl and Lesvia Elena~Ruiz Vel{\'{a}}zquez}.
\newblock \href{https://doi.org/10.1016/j.comgeo.2012.09.004}{Drawing planar 3-trees with given face areas}.
\newblock \emph{Comput. Geom.}, 46(3):276--285, 2013.

\bibitem[{Bienstock et~al.(1991)Bienstock, Robertson, Seymour, and Thomas}]{BRST91}
\textsc{Dan Bienstock, Neil Robertson, Paul Seymour, and Robin Thomas}.
\newblock \href{https://doi.org/10.1016/0095-8956(91)90068-U}{Quickly excluding a forest}.
\newblock \emph{J. Combin. Theory Ser. B}, 52(2):274--283, 1991.

\bibitem[{Bodlaender(1998)}]{Bodlaender98}
\textsc{Hans~L. Bodlaender}.
\newblock \href{https://doi.org/10.1016/S0304-3975(97)00228-4}{A partial $k$-arboretum of graphs with bounded treewidth}.
\newblock \emph{Theoret. Comput. Sci.}, 209(1-2):1--45, 1998.

\bibitem[{Chudnovsky et~al.(2024)Chudnovsky, Gartland, Hajebi, Lokshtanov, and Spirkl}]{CGHLS24}
\textsc{Maria Chudnovsky, Peter Gartland, Sepehr Hajebi, Daniel Lokshtanov, and Sophie Spirkl}.
\newblock \href{https://arxiv.org/abs/2402.14211}{Induced subgraphs and tree decompositions {XV}. {E}ven-hole-free graphs with bounded clique number have logarithmic treewidth}.
\newblock 2024, arXiv:2402.14211.

\bibitem[{Chuzhoy and Tan(2019)}]{CT19}
\textsc{Julia Chuzhoy and Zihan Tan}.
\newblock \href{https://doi.org/10.1137/1.9781611975482.88}{Towards tight(er) bounds for the excluded grid theorem}.
\newblock In \textsc{Timothy~M. Chan}, ed., \emph{Proc. 13th Annual {ACM-SIAM} Symp. Discrete Algorithms \textup{({SODA} '19)}}, pp. 1445--1464. 2019.

\bibitem[{Coudert et~al.(2016)Coudert, Ducoffe, and Nisse}]{CDN16}
\textsc{David Coudert, Guillaume Ducoffe, and Nicolas Nisse}.
\newblock \href{https://doi.org/10.1137/15M1034039}{To approximate treewidth, use treelength!}
\newblock \emph{{SIAM} J. Discrete Math.}, 30(3):1424--1436, 2016.

\bibitem[{Dallard et~al.(2024)Dallard, Milani\v{c}, and \v{S}torgel}]{DMS24a}
\textsc{Cl\'{e}ment Dallard, Martin Milani\v{c}, and Kenny \v{S}torgel}.
\newblock \href{https://doi.org/10.1016/j.jctb.2023.10.006}{Treewidth versus clique number. {II}. {T}ree-independence number}.
\newblock \emph{J. Combin. Theory Ser. B}, 164:404--442, 2024.

\bibitem[{Dehkordi and Farr(2021)}]{DF21}
\textsc{Hooman~Reisi Dehkordi and Graham Farr}.
\newblock \href{https://doi.org/10.37236/8816}{Non-separating planar graphs}.
\newblock \emph{Electron. J. Combin.}, 28(1):1, 2021.

\bibitem[{Diestel(2018)}]{Diestel5}
\textsc{Reinhard Diestel}.
\newblock Graph theory, vol. 173 of \emph{Graduate Texts in Mathematics}.
\newblock Springer, 5th edn., 2018.

\bibitem[{Diestel and M{\"u}ller(2016)}]{DM16}
\textsc{Reinhard Diestel and Malte M{\"u}ller}.
\newblock \href{https://doi.org/10.1017/S0963548315000287}{A short proof for lean tree-decompositions}.
\newblock \emph{Combin. Probab. Comput.}, 25(5):647--649, 2016.

\bibitem[{Diestel and M\"{u}ller(2018)}]{DM18}
\textsc{Reinhard Diestel and Malte M\"{u}ller}.
\newblock \href{https://doi.org/10.1007/s00493-016-3516-5}{Connected tree-width}.
\newblock \emph{Combinatorica}, 38(2):381--398, 2018.

\bibitem[{Dourisboure and Gavoille(2007)}]{DG07}
\textsc{Yon Dourisboure and Cyril Gavoille}.
\newblock \href{https://doi.org/10.1016/j.disc.2005.12.060}{Tree-decompositions with bags of small diameter}.
\newblock \emph{Discrete Math.}, 307(16):2008--2029, 2007.

\bibitem[{Dragan and Köhler(2025)}]{DK25}
\textsc{Feodor~F. Dragan and Ekkehard Köhler}.
\newblock \href{https://arxiv.org/abs/2503.05661}{Graph parameters that are coarsely equivalent to path-length}.
\newblock 2025, arXiv:2503.05661.

\bibitem[{Dujmovi{\'c} et~al.(2026)Dujmovi{\'c}, Morin, Norin, and Wood}]{DMNW}
\textsc{Vida Dujmovi{\'c}, Pat Morin, Sergey Norin, and David~R. Wood}.
\newblock \href{https://arxiv.org/abs/2507.03163v2}{3-{C}olouring planar graphs}.
\newblock 2026, arXiv:2507.03163.

\bibitem[{Erde and Wei{\ss}auer(2019)}]{EW19}
\textsc{Joshua Erde and Daniel Wei{\ss}auer}.
\newblock \href{https://doi.org/10.1137/18M1198296}{A short derivation of the structure theorem for graphs with excluded topological minors}.
\newblock \emph{SIAM J. Discrete Math.}, 33(3):1654--1661, 2019.

\bibitem[{Erd\H{o}s and Rado(1950)}]{ErdosRado50}
\textsc{Paul Erd\H{o}s and Richard Rado}.
\newblock \href{https://doi.org/10.1112/jlms/s1-25.4.249}{A combinatorial theorem}.
\newblock \emph{J. London Math. Society}, 25:249--255, 1950.

\bibitem[{Felsner et~al.(2003)Felsner, Liotta, and Wismath}]{FLW-JGAA03}
\textsc{Stefan Felsner, Giussepe Liotta, and Stephen~K. Wismath}.
\newblock \href{https://doi.org/10.7155/jgaa.00075}{Straight-line drawings on restricted integer grids in two and three dimensions}.
\newblock \emph{J. Graph Algorithms Appl.}, 7(4):363--398, 2003.

\bibitem[{Grigoriev(2011)}]{Grigoriev11}
\textsc{Alexander Grigoriev}.
\newblock \href{https://dmtcs.episciences.org/539}{Tree-width and large grid minors in planar graphs}.
\newblock \emph{Discrete Math. \& Theoret. Comput. Sci.}, 13(1):13--20, 2011.

\bibitem[{Harvey and Wood(2017)}]{HW17}
\textsc{Daniel~J. Harvey and David~R. Wood}.
\newblock \href{https://doi.org/10.1002/jgt.22030}{Parameters tied to treewidth}.
\newblock \emph{J. Graph Theory}, 84(4):364--385, 2017.

\bibitem[{Hendrey and Wood(2026)}]{HW26}
\textsc{Kevin Hendrey and David~R. Wood}.
\newblock \href{https://doi.org/10.1017/S0963548326100522}{Optimal tree-decompositions with bags of bounded treewidth}.
\newblock \emph{Combin. Probab. Comput.}, 2026.
\newblock arXiv:2511.22196.

\bibitem[{Hickingbotham(2025)}]{Hickingbotham25}
\textsc{Robert Hickingbotham}.
\newblock \href{https://arxiv.org/abs/2501.10840}{Graphs quasi-isometric to graphs with bounded treewidth}.
\newblock 2025, arXiv:2501.10840.

\bibitem[{Huynh et~al.(2022)Huynh, Joret, Micek, Seweryn, and Wollan}]{HJMSW22}
\textsc{Tony Huynh, Gwena{\"e}l Joret, Piotr Micek, Micha{\l}~T. Seweryn, and Paul Wollan}.
\newblock \href{https://doi.org/10.1007/s00493-021-4592-8}{Excluding a ladder}.
\newblock \emph{Combinatorica}, 42(3):405--432, 2022.

\bibitem[{Huynh and Kim(2017)}]{HK17}
\textsc{Tony Huynh and Ringi Kim}.
\newblock \href{https://doi.org/10.1002/jgt.22121}{Tree-chromatic number is not equal to path-chromatic number}.
\newblock \emph{J. Graph Theory}, 86(2):213--222, 2017.

\bibitem[{Huynh et~al.(2021)Huynh, Reed, Wood, and Yepremyan}]{HRWY21}
\textsc{Tony Huynh, Bruce Reed, David~R. Wood, and Liana Yepremyan}.
\newblock \href{https://doi.org/10.1007/978-3-030-62497-2_30}{Notes on tree- and path-chromatic number}.
\newblock In \emph{2019--20 {MATRIX} Annals}, vol.~4 of \emph{MATRIX Book Ser.}, pp. 489--498. Springer, 2021.

\bibitem[{Kostochka(1982)}]{Kostochka82}
\textsc{Alexandr~V. Kostochka}.
\newblock \href{https://kostochk.web.illinois.edu/docs/old/translation86.pdf}{The minimum {H}adwiger number for graphs with a given mean degree of vertices}.
\newblock \emph{Metody Diskret. Analiz.}, 38:37--58, 1982.

\bibitem[{Kostochka(1984)}]{Kostochka84}
\textsc{Alexandr~V. Kostochka}.
\newblock \href{https://doi.org/10.1007/BF02579141}{Lower bound of the {H}adwiger number of graphs by their average degree}.
\newblock \emph{Combinatorica}, 4(4):307--316, 1984.

\bibitem[{Koutsoutis et~al.(2026)Koutsoutis, Krause, Liu, Redzic, and Ueckerdt}]{KKLRU}
\textsc{Alex Koutsoutis, Kilian Krause, Chun-Hung Liu, Mirza Redzic, and Torsten Ueckerdt}.
\newblock \href{https://doi.org/10.4230/LIPIcs.WG.2026.31}{On the relation between treewidth, tree-independence number, and tree-chromatic number of graphs}.
\newblock In \textsc{Jan Goedgebeur and Pawe{\l} Rz\k{a}\.{z}ewski}, eds., \emph{Proc. 52nd International Workshop on Graph-Theoretic Concepts in Computer Science \textup{(WG 2026)}}, vol. 376 of \emph{LIPIcs}, pp. 31:1--31:9. Schloss Dagstuhl, 2026.

\bibitem[{Liu et~al.(2024)Liu, Norin, and Wood}]{LNW}
\textsc{Chun-Hung Liu, Sergey Norin, and David~R. Wood}.
\newblock \href{https://arxiv.org/abs/2410.20333}{Product structure and tree decompositions}.
\newblock 2024, arXiv:2410.20333.

\bibitem[{Mader(1968)}]{Mader68}
\textsc{Wolfgang Mader}.
\newblock \href{https://doi.org/10.1007/BF01350657}{Homomorphies\"atze f\"ur {G}raphen}.
\newblock \emph{Math. Ann.}, 178:154--168, 1968.

\bibitem[{Mohar and Thomassen(2001)}]{MoharThom}
\textsc{Bojan Mohar and Carsten Thomassen}.
\newblock Graphs on surfaces.
\newblock Johns Hopkins University Press, 2001.

\bibitem[{Myers and Thomason(2005)}]{MT05}
\textsc{Joseph~Samuel Myers and Andrew Thomason}.
\newblock \href{https://doi.org/10.1007/s00493-005-0044-0}{The extremal function for noncomplete minors}.
\newblock \emph{Combinatorica}, 25(6):725--753, 2005.

\bibitem[{Müller(2012)}]{Muller12}
\textsc{Malte Müller}.
\newblock \href{https://arxiv.org/abs/1211.7353v1}{Connected tree-width}.
\newblock 2012, arXiv:1211.7353v1.

\bibitem[{Ne{\v{s}}et{\v{r}}il and Ossona~de Mendez(2012)}]{Sparsity}
\textsc{Jaroslav Ne{\v{s}}et{\v{r}}il and Patrice Ossona~de Mendez}.
\newblock \href{https://doi.org/10.1007/978-3-642-27875-4}{Sparsity}, vol.~28 of \emph{Algorithms and Combinatorics}.
\newblock Springer, 2012.

\bibitem[{Nguyen et~al.(2025)Nguyen, Scott, and Seymour}]{NSS-AS-I}
\textsc{Tung Nguyen, Alex Scott, and Paul Seymour}.
\newblock \href{https://arxiv.org/abs/2501.09839}{Asymptotic structure. {I}. {C}oarse tree-width}.
\newblock 2025, arXiv:2501.09839.

\bibitem[{Norin et~al.(2020)Norin, Reed, Thomason, and Wood}]{NRTW20}
\textsc{Sergey Norin, Bruce Reed, Andrew Thomason, and David~R. Wood}.
\newblock \href{https://doi.org/10.37236/8847}{A lower bound on the average degree forcing a minor}.
\newblock \emph{Electron. J. Combin.}, 27:P2.4, 2020.

\bibitem[{Reed and Wood(2016)}]{ReedWood16}
\textsc{Bruce Reed and David~R. Wood}.
\newblock \href{https://doi.org/10.1017/S0963548315000073}{Forcing a sparse minor}.
\newblock \emph{Combin. Probab. Comput.}, 25(2):300--322, 2016.

\bibitem[{Reed(1997)}]{Reed97}
\textsc{Bruce~A. Reed}.
\newblock \href{https://doi.org/10.1017/CBO9780511662119.006}{Tree width and tangles: a new connectivity measure and some applications}.
\newblock In \textsc{R.~A. Bailey}, ed., \emph{Surveys in Combinatorics}, vol. 241 of \emph{London Math. Soc. Lecture Note Ser.}, pp. 87--162. Cambridge Univ. Press, 1997.

\bibitem[{Robertson and Seymour(1983)}]{RS-I}
\textsc{Neil Robertson and Paul Seymour}.
\newblock \href{https://doi.org/10.1016/0095-8956(83)90079-5}{Graph minors. {I}. {E}xcluding a forest}.
\newblock \emph{J. Combin. Theory Ser. B}, 35(1):39--61, 1983.

\bibitem[{Robertson and Seymour(1984)}]{RS-III}
\textsc{Neil Robertson and Paul Seymour}.
\newblock \href{https://doi.org/10.1016/0095-8956(84)90013-3}{Graph minors. {III}. {P}lanar tree-width}.
\newblock \emph{J. Combin. Theory Ser. B}, 36(1):49--64, 1984.

\bibitem[{Robertson and Seymour(1986)}]{RS-V}
\textsc{Neil Robertson and Paul Seymour}.
\newblock \href{https://doi.org/10.1016/0095-8956(86)90030-4}{Graph minors. {V}. {E}xcluding a planar graph}.
\newblock \emph{J. Combin. Theory Ser. B}, 41(1):92--114, 1986.

\bibitem[{Robertson et~al.(1994)Robertson, Seymour, and Thomas}]{RST94}
\textsc{Neil Robertson, Paul Seymour, and Robin Thomas}.
\newblock \href{https://doi.org/10.1006/jctb.1994.1073}{Quickly excluding a planar graph}.
\newblock \emph{J. Combin. Theory Ser. B}, 62(2):323--348, 1994.

\bibitem[{Scheffler(1989)}]{Scheffler89}
\textsc{Petra Scheffler}.
\newblock Die baumweite von graphen als ein ma{\ss} f{\"u}r die kompliziertheit algorithmischer probleme.
\newblock Ph.D. thesis, Akademie der Wissenschaften der DDR, Berlin, Germany, 1989.

\bibitem[{Seymour(2016)}]{Seymour16}
\textsc{Paul Seymour}.
\newblock \href{https://doi.org/10.1016/j.jctb.2015.08.002}{Tree-chromatic number}.
\newblock \emph{J. Combin. Theory Series B}, 116:229--237, 2016.

\bibitem[{Thomas(1990)}]{Thomas90}
\textsc{Robin Thomas}.
\newblock \href{https://doi.org/10.1016/0095-8956(90)90130-R}{A {M}enger-like property of tree-width: {T}he finite case}.
\newblock \emph{J. Combin. Theory Ser. B}, 48(1):67--76, 1990.

\bibitem[{Thomason(1984)}]{Thomason84}
\textsc{Andrew Thomason}.
\newblock \href{https://doi.org/10.1017/S0305004100061521}{An extremal function for contractions of graphs}.
\newblock \emph{Math. Proc. Cambridge Philos. Soc.}, 95(2):261--265, 1984.

\bibitem[{Thomason(2001)}]{Thomason01}
\textsc{Andrew Thomason}.
\newblock \href{https://doi.org/10.1006/jctb.2000.2013}{The extremal function for complete minors}.
\newblock \emph{J. Combin. Theory Ser. B}, 81(2):318--338, 2001.

\bibitem[{Thomason and Wales(2022)}]{TW22}
\textsc{Andrew Thomason and Matthew Wales}.
\newblock \href{https://doi.org/10.1002/jgt.22811}{On the extremal function for graph minors}.
\newblock \emph{J. Graph Theory}, 101(1):66--78, 2022.

\bibitem[{Wei{\ss}auer(2019)}]{Weissauer19}
\textsc{Daniel Wei{\ss}auer}.
\newblock \href{https://doi.org/10.1137/17M1145306}{On the block number of graphs}.
\newblock \emph{SIAM J. Discrete Math.}, 33(1):346--357, 2019.

\bibitem[{Yolov(2018)}]{Yolov18}
\textsc{Nikola Yolov}.
\newblock \href{https://doi.org/10.1137/1.9781611975031.16}{Minor-matching hypertree width}.
\newblock In \textsc{Artur Czumaj}, ed., \emph{Proc. 29th Annual {ACM-SIAM} Symposium on Discrete Algorithms \textup{(SODA 2018)}}, pp. 219--233. {SIAM}, 2018.

\end{thebibliography}
\end{document}